\documentclass[reqno]{amsart}

\usepackage{hyperref}

\usepackage{longtable}

\newtheorem{theorem}{Theorem}
\newtheorem{lemma}[theorem]{Lemma}
\newtheorem{proposition}[theorem]{Proposition}
\newtheorem{corollary}[theorem]{Corollary}
\newtheorem{example}[theorem]{Example}
\theoremstyle{exercise}
\newtheorem{exercise}[theorem]{Exercise}
\theoremstyle{definition}
\newtheorem{definition}[theorem]{Definition}
\usepackage{graphicx}
\usepackage{pstricks, enumerate, pst-node, pst-text, pst-plot}
\theoremstyle{remark}
\newtheorem{remark}[theorem]{Remark}

\numberwithin{equation}{section}

\newcommand{\intav}[1]{\mathchoice {\mathop{\vrule width 6pt height 3 pt depth  -2.5pt
\kern -8pt \intop}\nolimits_{\kern -6pt#1}} {\mathop{\vrule width
5pt height 3  pt depth -2.6pt \kern -6pt \intop}\nolimits_{#1}}
{\mathop{\vrule width 5pt height 3 pt depth -2.6pt \kern -6pt
\intop}\nolimits_{#1}} {\mathop{\vrule width 5pt height 3 pt depth
-2.6pt \kern -6pt \intop}\nolimits_{#1}}}

\newcommand{\intavl}[1]{\mathchoice {\mathop{\vrule width 6pt height 3 pt depth  -2.5pt
\kern -8pt \intop}\limits_{\kern -6pt#1}} {\mathop{\vrule width 5pt
height 3  pt depth -2.6pt \kern -6pt \intop}\nolimits_{#1}}
{\mathop{\vrule width 5pt height 3 pt depth -2.6pt \kern -6pt
\intop}\nolimits_{#1}} {\mathop{\vrule width 5pt height 3 pt depth
-2.6pt \kern -6pt \intop}\nolimits_{#1}}}

\newcommand{\R}{\mathbb{R}}
\newcommand{\N}{\mathbb{N}}

\newcommand{\Z}{\mathbb{Z}}

\begin{document}

\title[Lagrange and Markov spectra from the dynamical viewpoint]{The Lagrange and Markov spectra from the dynamical point of view}

\author[Carlos Matheus]{Carlos Matheus}
\address{Carlos Matheus: Universit\'e Paris 13, Sorbonne Paris Cit\'e, LAGA, CNRS (UMR 7539), F-93439, Villetaneuse, France.}
\email{matheus@impa.br}

%\author[Carlos Matheus]{Carlos Matheus}
%\address{College de France, 3 Rue d'Ulm, Paris CEDEX 05, France.}
%\email{matheus@impa.br}

%    General info
%\subjclass[??]{??}

\date{\today}

%\keywords{Continued fractions, Diophantine approximations, Gauss-Cantor sets, Hausdorff dimension, %Lagrange spectrum, Markov spectrum, Moreira's theorem.}

\begin{abstract}
This text grew out of my lecture notes for a 4-hours minicourse delivered on October 17 \& 19, 2016 during the research school ``Applications of Ergodic Theory in Number Theory'' -- an activity related to the Jean-Molet Chair project of Mariusz Lema\'nczyk and S\'ebastien Ferenczi -- realized at CIRM, Marseille, France. The subject of this text is the same of my minicourse, namely, the structure of the so-called Lagrange and Markov spectra (with an special emphasis on a recent theorem of C. G. Moreira). 
\end{abstract}

\maketitle

\tableofcontents

%%%%%%%%%%%%%%%%%%%%%%%%%%%%%%%%%%%
%%%%%%%%%%%%%%%%%%%%%%%%%%%%%%%%%%%
%%%%%%%%%%%%%%%%%%%%%%%%%%%%%%%%%%%
%%%%%%%%%%%%%%%%%%%%%%%%%%%%%%%%%%%
%%%%%%%%%%%% Section 1 %%%%%%%%%%%%%%%%
%%%%%%%% Classical spectra %%%%%%%%%%%%%%
%%%%%%%%%%%%%%%%%%%%%%%%%%%%%%%%%%%
%%%%%%%%%%%%%%%%%%%%%%%%%%%%%%%%%%%
%%%%%%%%%%%%%%%%%%%%%%%%%%%%%%%%%%%
%%%%%%%%%%%%%%%%%%%%%%%%%%%%%%%%%%%

\section{Diophantine approximations \& Lagrange and Markov spectra}

\subsection{Rational approximations of real numbers}  

Given a real number $\alpha\in\mathbb{R}$, it is natural to compare the quality $|\alpha-p/q|$ of a rational approximation $p/q\in\mathbb{Q}$ and the size $q$ of its denominator. 

Since any real number lies between two consecutive integers, for every $\alpha\in\mathbb{R}$ and $q\in\mathbb{N}$, there exists $p\in\mathbb{Z}$ such that $|q\alpha-p|\leq 1/2$, i.e. 
\begin{equation}\label{e.pre-Dirichlet}
\left|\alpha-\frac{p}{q}\right|\leq \frac{1}{2q}
\end{equation}

In 1842, Dirichlet \cite{Di} used his famous \emph{pigeonhole principle} to improve \eqref{e.pre-Dirichlet}.

\begin{theorem}[Dirichlet] For any $\alpha\in\mathbb{R}-\mathbb{Q}$, the inequality 
$$\left|\alpha-\frac{p}{q}\right|\leq\frac{1}{q^2}$$ 
has infinitely many rational solutions $p/q\in\mathbb{Q}$.
\end{theorem}

\begin{proof} Given $Q\in\mathbb{N}$, we decompose the interval $[0,1)$ into $Q$ disjoint subintervals as follows:
$$[0,1)=\bigcup\limits_{j=0}^{Q-1} \left[\frac{j}{Q}, \frac{j+1}{Q}\right)$$
Next, we consider the $Q+1$ distinct\footnote{$\alpha\notin\mathbb{Q}$ is used here} numbers $\{i\alpha\}$, $i=0, \dots, Q$, where $\{x\}$ denotes the \emph{fractional part}\footnote{$\{x\}:=x-\lfloor x\rfloor$ and $\lfloor x\rfloor:=\max\{n\in\mathbb{Z}: n\leq x\}$ is the integer part of $x$.} of $x$. By the \emph{pigeonhole principle}, some interval $\left[\frac{j}{Q}, \frac{j+1}{Q}\right)$ must contain two such numbers, say $\{n\alpha\}$ and $\{m\alpha\}$, $0\leq n < m\leq Q$. It follows that 
$$|\{m\alpha\}-\{n\alpha\}|<\frac{1}{Q},$$
i.e., $|q\alpha-p|<1/Q$ where $0<q:=m-n\leq Q$ and $p:=\lfloor m\alpha\rfloor - \lfloor n\alpha\rfloor$. Therefore, 
$$\left|\alpha-\frac{p}{q}\right|<\frac{1}{qQ}\leq\frac{1}{q^2}$$ 
This completes the proof of the theorem. 
\end{proof}

In 1891, Hurwitz \cite{Hu} showed that Dirichlet's theorem is essentially optimal: 

\begin{theorem}[Hurwitz] For any $\alpha\in\mathbb{R}-\mathbb{Q}$, the inequality 
$$\left|\alpha-\frac{p}{q}\right|\leq\frac{1}{\sqrt{5}q^2}$$ 
has infinitely many rational solutions $p/q\in\mathbb{Q}$. 

Moreover, for all $\varepsilon>0$, the inequality 
$$\left|\frac{1+\sqrt{5}}{2}-\frac{p}{q}\right|\leq\frac{1}{(\sqrt{5}+\varepsilon)q^2}$$ 
has only finitely many rational solutions $p/q\in\mathbb{Q}$. 
\end{theorem}

The first part of Hurwitz theorem is proved in Appendix \ref{a.Hurwitz}, while the second part of Hurwitz theorem is left as an exercise to the reader:  

\begin{exercise} Show the second part of Hurwitz theorem. (Hint: use the identity $p^2-pq-q^2 = \left(q\frac{1+\sqrt{5}}{2} - p\right) \left(q\frac{1-\sqrt{5}}{2}-p\right)$ relating $\frac{1+\sqrt{5}}{2}$ and its Galois conjugate $\frac{1-\sqrt{5}}{2}$). 

Moreover, use your argument to give a bound on 
$$\#\left\{\frac{p}{q}\in\mathbb{Q}: \left|\frac{1+\sqrt{5}}{2} - \frac{p}{q} \right|\leq \frac{1}{(\sqrt{5}+\varepsilon)q^2}\right\}$$ in terms of $\varepsilon>0$. 
\end{exercise}

Note that Hurwitz theorem does \emph{not} forbid an improvement of ``$\left|\alpha-\frac{p}{q}\right|\leq \frac{1}{\sqrt{5} q^2}$ has infinitely many rational solutions $p/q\in\mathbb{Q}$'' for \emph{certain} $\alpha\in\mathbb{R}-\mathbb{Q}$. This motivates the following definition:

\begin{definition} The constant 
$$\ell(\alpha) := \limsup\limits_{p, q\to\infty} \frac{1}{|q(q\alpha-p)|}$$ 
is called the \emph{best constant of Diophantine approximation} of $\alpha$. 
\end{definition} 

Intuitively, $\ell(\alpha)$ is the best constant $\ell$ such that $|\alpha-\frac{p}{q}|\leq \frac{1}{\ell q^2}$
has infinitely many rational solutions $p/q\in\mathbb{Q}$. 

\begin{remark} By Hurwitz theorem, $\ell(\alpha)\geq\sqrt{5}$ for all $\alpha\in\mathbb{R}-\mathbb{Q}$ and $\ell(\frac{1+\sqrt{5}}{2})=\sqrt{5}$. 
\end{remark}

The collection of \emph{finite} best constants of Diophantine approximations is the \emph{Lagrange spectrum}: 

\begin{definition} The \emph{Lagrange spectrum} is 
$$L:=\{\ell(\alpha): \alpha\in\mathbb{R}-\mathbb{Q}, \ell(\alpha)<\infty\}\subset\mathbb{R}$$ 
\end{definition}

\begin{remark} Khinchin proved in 1926 a famous theorem implying that $\ell(\alpha)=\infty$ for Lebesgue almost every $\alpha\in\mathbb{R}-\mathbb{Q}$ (see, e.g., Khinchin's book \cite{Kh} for more details).  
\end{remark} 

\subsection{Integral values of binary quadratic forms} 

Let $q(x,y)=ax^2+bxy+cy^2$ be a \emph{binary quadratic form} with real coefficients $a, b, c\in \mathbb{R}$. Suppose that $q$ is \emph{indefinite}\footnote{I.e., $q$ takes both positive and negative values.} with positive \emph{discriminant} $\Delta(q):=b^2-4ac$. What is the smallest value of $q(x,y)$ at non-trivial integral vectors $(x,y)\in\mathbb{Z}^2-\{(0,0)\}$? 

\begin{definition} The \emph{Markov spectrum} is 
$$M:=\left\{\frac{\sqrt{\Delta(q)}}{\inf\limits_{(x,y)\in\mathbb{Z}^2-\{(0,0)\}}|q(x,y)|}\in\mathbb{R}: q \textrm{ is an indefinite binary quadratic form with } \Delta(q)>0\right\}$$
\end{definition}

\begin{remark} A similar Diophantine problem for \emph{ternary} (and $n$-ary, $n\geq 3$) quadratic forms was proposed by Oppenheim in 1929. Oppenheim's conjecture was famously solved in 1987 by Margulis using \emph{dynamics on homogeneous spaces}: the reader is invited to consult Witte Morris book \cite{WM} for more details about this beautiful portion of Mathematics. 
\end{remark}

In 1880, Markov \cite{Ma} noticed a relationship between certain binary quadratic forms and rational approximations of certain irrational numbers. This allowed him to prove the following result:

\begin{theorem}[Markov]\label{t.Markov} $L \cap (-\infty, 3) = M \cap (-\infty, 3) = \{k_1<k_2<k_3<k_4<\dots\}$
 where $k_1=\sqrt{5}$, $k_2=\sqrt{8}$, $k_3=\frac{\sqrt{221}}{5}$, $k_4=\frac{\sqrt{1517}}{13}$, $\dots$ is an explicit increasing sequence of quadratic surds\footnote{I.e., $k_n^2\in\mathbb{Q}$ for all $n\in\mathbb{N}$.} accumulating at $3$. 

In fact, $k_n=\sqrt{9-\frac{4}{m_n^2}}$ where $m_n\in\mathbb{N}$ is the $n$-th Markov number, and a Markov number is the largest coordinate of a Markov triple $(x,y,z)$, i.e., an integral solution of $x^2+y^2+z^2=3xyz$.
\end{theorem} 

\begin{remark} All Markov triples can be deduced from $(1,1,1)$ by applying the so-called \emph{Vieta involutions} $V_1, V_2, V_3$ given by 
$$V_1(x,y,z) = (x',y,z)$$
where $x'=3yz-x$ is the other solution of the second degree equation $X^2-3yzX+(y^2+z^2)=0$, etc. In other terms, all Markov triples appear in \emph{Markov tree}\footnote{Namely, the tree where Markov triples $(x,y,z)$ are displayed after applying permutations to put them in normalized form $x\leq y\leq z$,  and two normalized Markov triples are connected if we can obtain one from the other by applying Vieta involutions.}: 
\begin{figure}[htb!]
\includegraphics[scale=0.4]{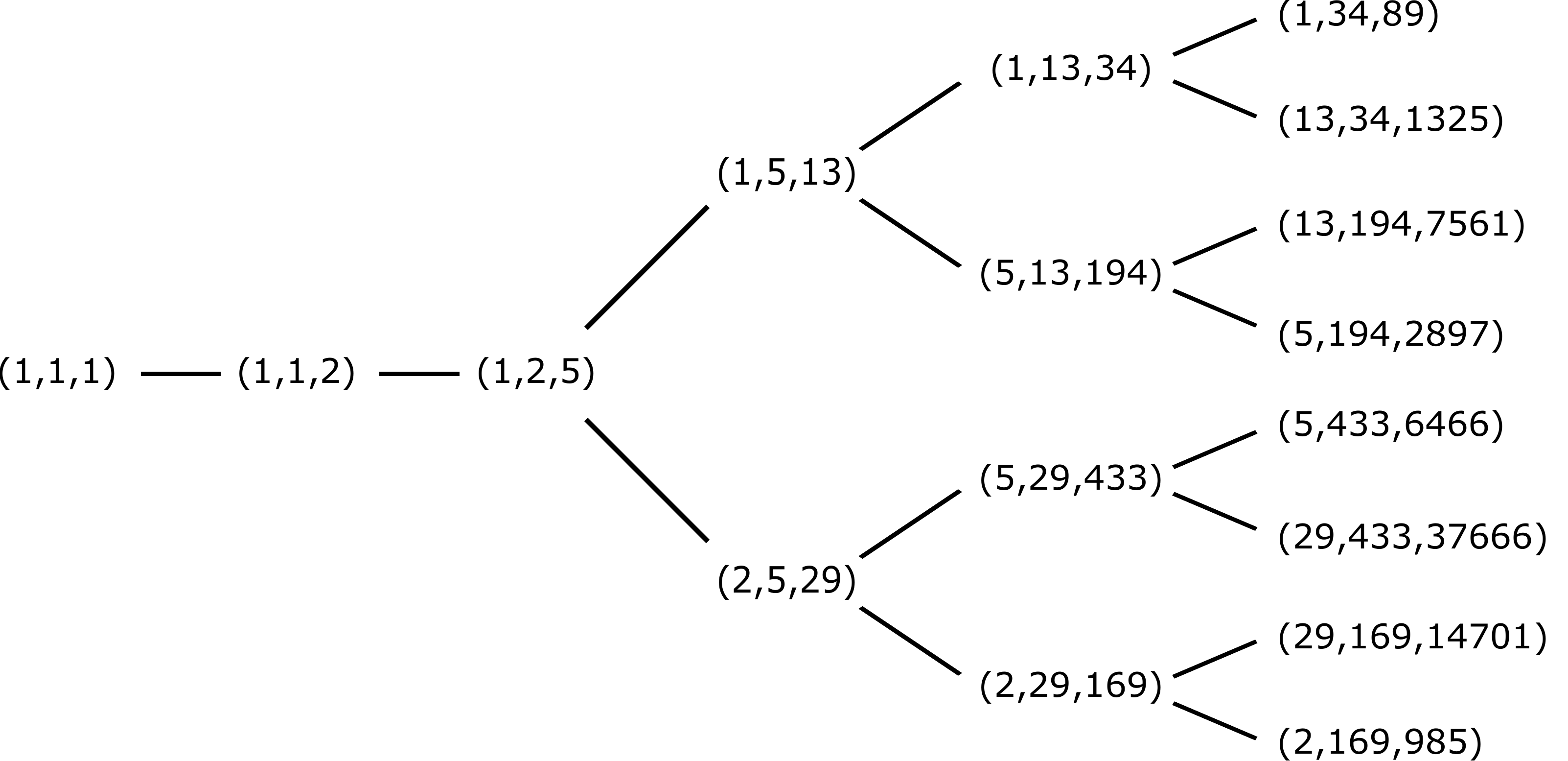}
\end{figure} 
\end{remark}

\begin{remark} For more informations on Markov numbers, the reader might consult Zagier's paper \cite{Za} on this subject. Among many conjectures and results mentioned in this paper, we have:
\begin{itemize}
\item Conjecturally, each Markov number $z$ determines \emph{uniquely} Markov triples $(x,y,z)$ with $x\leq y\leq z$; 
\item If $M(x):=\#\{m \textrm{ Markov number}: m\leq x\}$, then $M(x)=c(\log x)^2 + O(\log x (\log\log x)^2)$ for an \emph{explicit} constant $c\simeq 0.18071704711507...$; conjecturally, $M(x)=c(\log(3x))^2 + o(\log x)$, i.e., if $m_n$ is the $n$-th Markov number (counted with multiplicity), then $m_n\sim \frac{1}{3} A^{\sqrt{n}}$ with $A=e^{1/\sqrt{c}}\simeq 10.5101504...$
\end{itemize} 
\end{remark}

\subsection{Best rational approximations and continued fractions}

The constant $\ell(\alpha)$ was defined in terms of rational approximations of $\alpha\in\mathbb{R}-\mathbb{Q}$. In particular, 
$$\ell(\alpha)=\limsup\limits_{n\to\infty}\frac{1}{|s_n(s_n\alpha-r_n)|}$$
where $(r_n/s_n)_{n\in\mathbb{N}}$ is the sequence of best rational approximations of $\alpha$. Here, $p/q$ is called a \emph{best rational approximation}\footnote{This nomenclature will be justified later by Propositions \ref{p.finding-best} and \ref{p.best-nomenclature} below.} whenever 
$$\left|\alpha-\frac{p}{q}\right|<\frac{1}{2q^2}$$ 

The sequence $(r_n/s_n)_{n\in\mathbb{N}}$ of best rational approximations of $\alpha$ is produced by the so-called \emph{continued fraction algorithm}. 

Given $\alpha=\alpha_0\notin\mathbb{Q}$, we define recursively $a_n=\lfloor\alpha_n\rfloor$ and $\alpha_{n+1} = \frac{1}{\alpha_n-a_n}$ for all $n\in\mathbb{N}$. We can write $\alpha$ as a \emph{continued fraction} 
$$\alpha=a_0+\frac{1}{a_1+\frac{1}{a_2+\frac{1}{\ddots}}}=:[a_0; a_1, a_2,\dots]$$ 
and we denote 
$$\mathbb{Q}\ni\frac{p_n}{q_n}:=a_0+\frac{1}{a_1+\frac{1}{\ddots+\frac{1}{a_n}}}:=[a_0; a_1,\dots, a_n]$$

\begin{remark} L\'evy's theorem \cite{Le} (from 1936) says that $\sqrt[n]{q_n}\to e^{\pi^2/12\log 2}\simeq 3.27582291872...$ for Lebesgue almost every $\alpha\in\mathbb{R}$. By elementary properties of continued fractions (recalled below), it follows from L\'evy's theorem that $\sqrt[n]{|\alpha-\frac{p_n}{q_n}|}\to e^{-\pi^2/6\log 2}\simeq 0.093187822954...$ for Lebesgue almost every $\alpha\in\mathbb{R}$. 
\end{remark}

\begin{proposition}\label{p.pnqn} $p_n$ and $q_n$ are recursively given by 
$$\left\{\begin{array}{cc} p_{n+2} = a_{n+2} p_{n+1} + p_n, & p_{-1} = 1, p_{-2} = 0 \\ q_{n+2} = a_{n+2} q_{n+1} + q_n, & q_{-1} = 0, q_{-2} = 1 \end{array}\right.$$
\end{proposition} 

\begin{proof} Exercise\footnote{Hint: Use induction and the fact that $[t_0; t_1,\dots, t_n, t_{n+1}] = [t_0; t_1,\dots, t_n+\frac{1}{t_{n+1}}]$.}.
\end{proof} 

In other words, we have 
\begin{equation}\label{e.Mobius-fraction}
[a_0; a_1,\dots, a_{n-1}, z] = \frac{zp_{n-1}+p_{n-2}}{zq_{n-1}+q_{n-2}}
\end{equation} 
or, equivalently, 
\begin{equation}\label{e.SL2Z-fraction}
\left(\begin{array}{cc} p_{n+1} & p_n \\ q_{n+1} & q_n \end{array}\right) \cdot 
\left(\begin{array}{cc} a_{n+2} & 1 \\ 1 & 0 \end{array}\right) = 
\left(\begin{array}{cc} p_{n+2} & p_{n+1} \\ q_{n+2} & q_{n+1} \end{array}\right)
\end{equation}

\begin{corollary}\label{c.determinant-fraction} $p_{n+1}q_n-p_nq_{n+1} = (-1)^n$ for all $n\geq 0$.
\end{corollary}

\begin{proof} This follows from \eqref{e.SL2Z-fraction} because the matrix $\left(\begin{array}{cc}\ast & 1 \\ 1 & 0 \end{array}\right)$ has determinant $-1$. 
\end{proof} 

\begin{corollary}\label{c.Mobius-fraction} $\alpha=\frac{\alpha_n p_{n-1} + p_{n-2}}{\alpha_n q_{n-1} + q_{n-2}}$ and $\alpha_n = \frac{p_{n-2} - q_{n-2}\alpha}{q_{n-1}\alpha-p_{n-1}}$. 
\end{corollary} 

\begin{proof} This is a consequence of \eqref{e.Mobius-fraction} and the fact that $\alpha =: [a_0; a_1,\dots, a_{n-1},\alpha_n]$.
\end{proof}

The relationship between $\frac{p_n}{q_n}$ and the sequence of best rational approximations is explained by the following two propositions: 

\begin{proposition} $\left|\alpha-\frac{p_n}{q_n}\right|\leq \frac{1}{q_n q_{n+1}} < \frac{1}{a_{n+1} q_n^2}\leq \frac{1}{q_n^2}$ and, moreover, for all $n\in\mathbb{N}$, 
$$\textrm{either } \left|\alpha-\frac{p_n}{q_n}\right| < \frac{1}{2q_n^2} \textrm{ or } \left|\alpha-\frac{p_{n+1}}{q_{n+1}}\right|< \frac{1}{2q_{n+1}^2}.$$
\end{proposition}

\begin{proof} Note that $\alpha$ belongs to the interval with extremities $p_n/q_n$ and $p_{n+1}/q_{n+1}$ (by Corollary \ref{c.Mobius-fraction}). Since this interval has size 
$$\left|\frac{p_{n+1}}{q_{n+1}} - \frac{p_n}{q_n}\right| = \left|\frac{p_{n+1}q_n - p_n q_{n+1}}{q_nq_{n+1}} \right| = \left|\frac{(-1)^n}{q_n q_{n+1}}\right| = \frac{1}{q_n q_{n+1}}$$
(by Corollary \ref{c.determinant-fraction}), we conclude that $|\alpha-\frac{p_n}{q_n}|\leq \frac{1}{q_n q_{n+1}}$. 

Furthermore, $\frac{1}{q_n q_{n+1}} = |\frac{p_{n+1}}{q_{n+1}}-\alpha| + |\alpha-\frac{p_n}{q_n}|$. Thus, if $$\left|\alpha-\frac{p_n}{q_n}\right| \geq \frac{1}{2q_n^2} \quad \textrm{ and } \quad \left|\alpha-\frac{p_{n+1}}{q_{n+1}}\right|\geq \frac{1}{2q_{n+1}^2},$$
then 
$$\frac{1}{q_n q_{n+1}}\geq \frac{1}{2q_n^2}+\frac{1}{2q_{n+1}^2},$$
i.e., $2q_n q_{n+1}\geq q_n^2+q_{n+1}^2$, i.e., $q_n=q_{n+1}$, a contradiction. 
\end{proof}

In other terms, the sequence $(p_n/q_n)_{n\in\mathbb{N}}$ produced by the continued fraction algorithm contains best rational approximations with frequency at least $1/2$. 

Conversely, the continued fraction algorithm detects \emph{all} best rational approximations: 

\begin{proposition}\label{p.finding-best} If $|\alpha-\frac{p}{q}|<\frac{1}{2q^2}$, then $p/q = p_n/q_n$ for some $n\in\mathbb{N}$. 
\end{proposition} 

\begin{proof} Exercise\footnote{Hint: Take $q_{n-1}< q\leq q_n$, suppose that $p/q\neq p_n/q_n$ and derive a contradiction in each case $q=q_n$, $q_n/2\leq q<q_n$ and $q<q_n/2$ by analysing $|\alpha-\frac{p}{q}|$ and $|\frac{p}{q}-\frac{p_n}{q_n}|$ like in the proof of Proposition \ref{p.best-nomenclature}.}.
\end{proof}

The terminology ``best rational approximation'' is motivated by the previous proposition and the following result: 

\begin{proposition}\label{p.best-nomenclature} For all $q < q_n$, we have $|\alpha-\frac{p_n}{q_n}| < |\alpha-\frac{p}{q}|$. 
\end{proposition}

\begin{proof} If $q<q_{n+1}$ and $p/q\neq p_n/q_n$, then 
$$\left|\frac{p}{q}-\frac{p_n}{q_n}\right|\geq \frac{1}{q q_n} > \frac{1}{q_n q_{n+1}} = \left|\frac{p_{n+1}}{q_{n+1}}-\frac{p_n}{q_n}\right|$$
Hence, $p/q$ does not belong to the interval with extremities $p_n/q_n$ and $p_{n+1}/q_{n+1}$, and so 
$$\left|\alpha-\frac{p_n}{q_n}\right|< \left|\alpha-\frac{p}{q}\right|$$ 
because $\alpha$ lies between $p_n/q_n$ and $p_{n+1}/q_{n+1}$. 
\end{proof}

In fact, the approximations $(p_n/q_n)$ of $\alpha$ are usually quite impressive: 

\begin{example} $\pi=[3; 7, 15, 1, 292, 1, 1, 1, 2, 1, 3, 1, 14, 2, 1, \dots]$ so that 
$$\frac{p_0}{q_0}=3, \quad \frac{p_1}{q_1}=\frac{22}{7}, \quad \frac{p_2}{q_2}=\frac{333}{106}, \quad \frac{p_3}{q_3}=\frac{355}{113}, \quad \dots$$
The approximations $p_1/q_1$ and $p_3/q_3$ are called Yuel\"u and Mil\"u (after Wikipedia) and they are somewhat spectacular: 
$$\left|\pi-\frac{22}{7}\right|<\frac{1}{700}<\left|\pi-\frac{314}{100}\right| \quad \textrm{ and } \quad \left|\pi-\frac{355}{113}\right|<\frac{1}{3,000,000}<\left|\pi-\frac{3141592}{1,000,000}\right|$$
\end{example} 

\subsection{Perron's characterization of Lagrange and Markov spectra}

In 1921, Perron interpreted $\ell(\alpha)$ in terms of Dynamical Systems as follows. 

\begin{proposition}\label{p.Perron} $\alpha-\frac{p_n}{q_n} = \frac{(-1)^n}{(\alpha_{n+1}+\beta_{n+1})q_n^2}$ where $\beta_{n+1}:=\frac{q_{n-1}}{q_n}=[0; a_n, a_{n-1}, \dots, a_1]$. 
\end{proposition}

\begin{proof} Recall that $\alpha_{n+1} = \frac{p_{n-1}-q_{n-1}\alpha}{q_n\alpha-p_n}$ (cf. Corollary \ref{c.Mobius-fraction}). Hence, $\alpha_{n+1}+\beta_{n+1} = \frac{p_{n-1}q_n-p_nq_{n-1}}{q_n(q_n\alpha-p_n)} = \frac{(-1)^n}{q_n(q_n\alpha-p_n)}$ (by Corollary \ref{c.determinant-fraction}). This proves the proposition. 
\end{proof}

Therefore, the proposition says that $\ell(\alpha)=\limsup\limits_{n\to\infty}(\alpha_n+\beta_n)$. From the dynamical point of view, we consider the \emph{symbolic space} $\Sigma = (\mathbb{N}^*)^{\mathbb{Z}}=: \Sigma^-\times\Sigma^+ = (\mathbb{N}^*)^{\mathbb{Z}^-}\times (\mathbb{N}^*)^{\mathbb{N}}$ equipped with the left \emph{shift dynamics} $\sigma:\Sigma\to\Sigma$, $\sigma((a_n)_{n\in\mathbb{Z}}):=(a_{n+1})_{n\in\mathbb{Z}}$ and the \emph{height function} $f:\Sigma\to\mathbb{R}$, $f((a_n)_{n\in\mathbb{Z}}) = [a_0; a_1, a_2, \dots] + [0; a_{-1}, a_{-2}, \dots]$. Then, the proposition above implies that 
$$\ell(\alpha) = \limsup\limits_{n\to+\infty} f(\sigma^n(\underline{\theta}))$$
where $\alpha=[a_0; a_1, a_2, \dots]$ and $\underline{\theta} = (\dots, a_{-1}, a_0, a_1, \dots)$. In particular, 
\begin{equation}\label{e.Perron-Lagrange}
L = \{\ell(\underline{\theta}): \underline{\theta}\in\Sigma, \ell(\underline{\theta})<\infty\} 
\end{equation} 
where $\ell(\underline{\theta}):=\limsup\limits_{n\to+\infty} f(\sigma^n(\underline{\theta}))$. 

Also, the Markov spectrum has a \emph{similar description}: 
\begin{equation}\label{e.Perron-Markov}
M=\{m(\underline{\theta}): \underline{\theta}\in\Sigma, m(\underline{\theta})<\infty\}
\end{equation} 
where $m(\underline{\theta}):=\sup\limits_{n\in\mathbb{Z}} f(\sigma^n(\underline{\theta}))$. 

\begin{remark}\label{r.Perron-Gauss} A \emph{geometrical interpretation} of $\sigma:\Sigma\to\Sigma$ is provided by the so-called \emph{Gauss map}\footnote{From Number Theory rather than Differential Geometry.}:
\begin{equation}\label{e.Gauss-map} 
G(x)=\left\{\frac{1}{x}\right\}
\end{equation}
for $0<x\leq 1$. 

\begin{figure}[htb!]
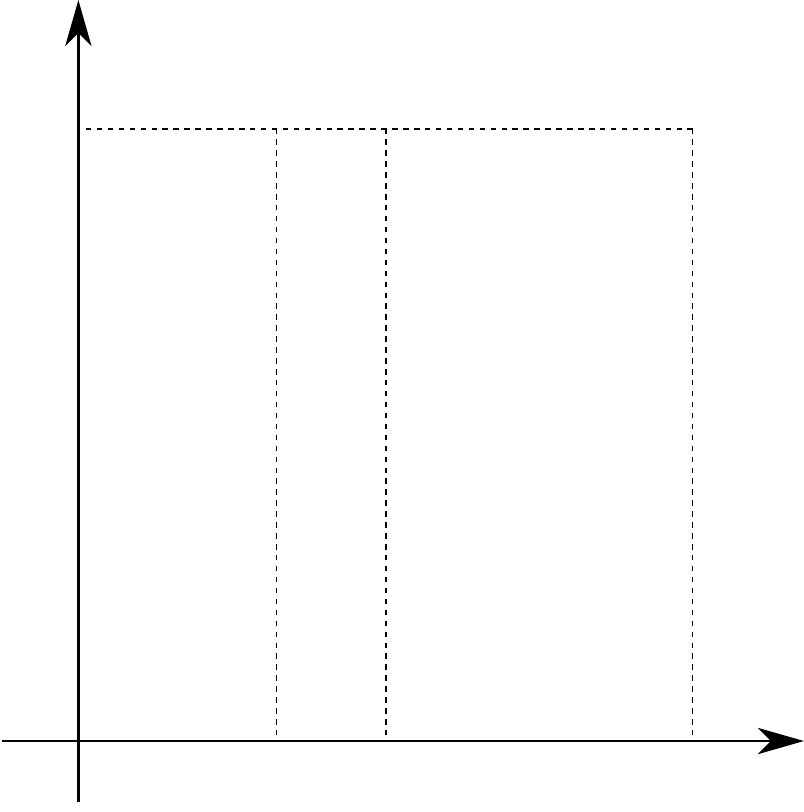
\end{figure}

Indeed, $G([0; a_1, a_2, \dots]) = [0; a_2, \dots]$, so that $\sigma:\Sigma\to\Sigma$ is a symbolic version of the \emph{natural extension} of $G$. 

Furthermore, the identification $(\dots, a_{-1}, a_0, a_1, \dots)\simeq ([0; a_{-1}, a_{-2}, \dots], [a_0; a_1, a_2, \dots]) = (y,x)$ allows us to write the height function as $f((a_n)_{n\in\mathbb{Z}}) = x+y$. 
\begin{figure}[htb!]
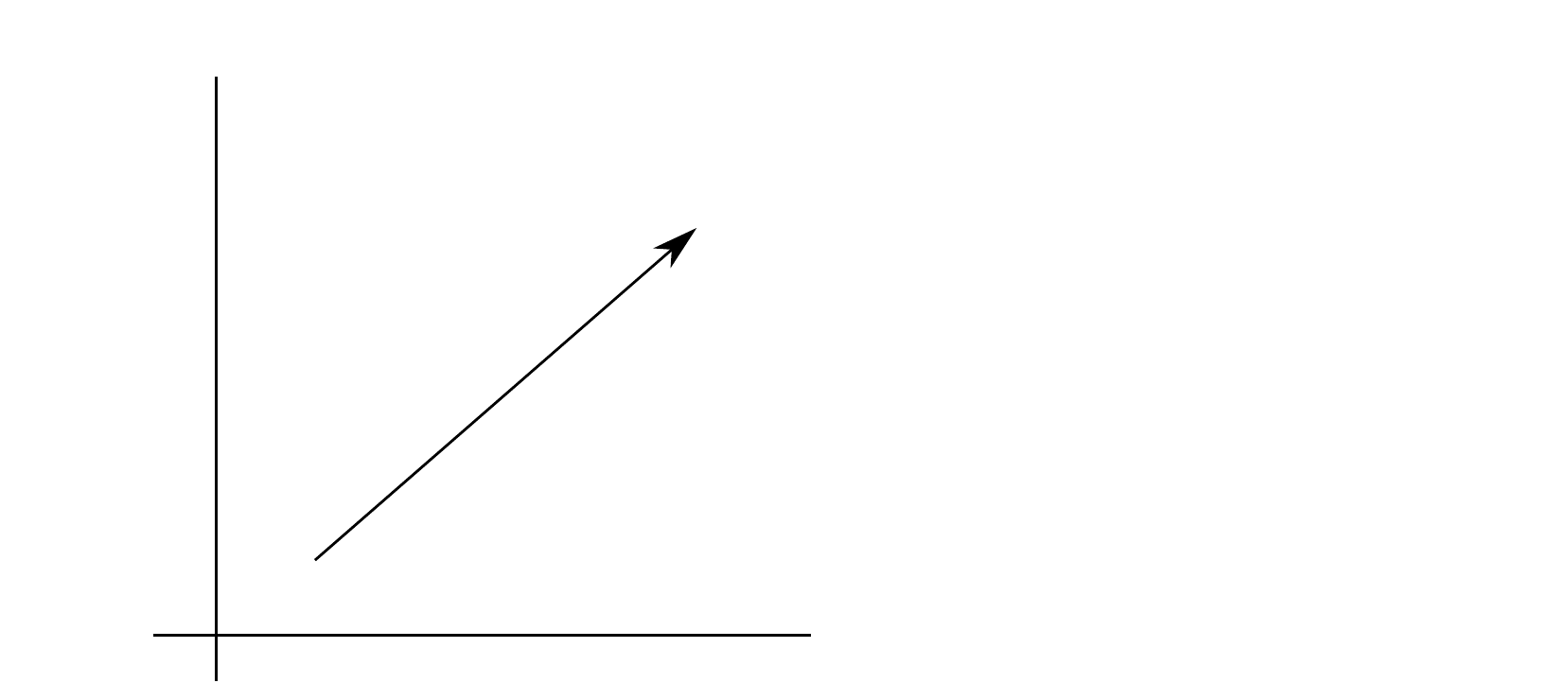
\end{figure}
\end{remark}

Perron's dynamical interpretation of the Lagrange and Markov spectra is the starting point of many results about $L$ and $M$ which are not so easy to guess from their definitions: 

\begin{exercise} Show that $L\subset M$ are closed subsets of $\mathbb{R}$. 
\end{exercise}

\begin{remark}\label{r.Freiman-M-L} $M-L\neq\emptyset$: for example, Freiman \cite{Fr68} proved in 1968 that 
$$s=\overline{221221122}11\overline{221122122}\in(\mathbb{N}^*)^{\mathbb{Z}}$$
has the property that $3.118120178\simeq m(s)\in M-L$. (Here $\overline{\theta_1\dots\theta_n}$ means infinite repetition of the block $\theta_1\dots\theta_n$.) 

Also, Freiman \cite{Fr73} showed in 1973 that $m(s_n)\in M-L$ and $m(s_n)\to m(s_{\infty})\simeq 3.293044265 \in M-L$ where 
$$s_n = \overline{2221121}\underbrace{22\dots 22}_{n \textrm{ times}}121122212\overline{1122212}$$
for $n\geq 4$, and 
$$s_{\infty} = \overline{2}121122212\overline{1122212}$$
\end{remark}

\subsection{Digression: Lagrange spectrum and cusp excursions on the modular surface} 

The Lagrange spectrum is related to the values of a certain height function $H$ along the orbits of the geodesic flow $g_t$ on the (unit cotangent bundle to) the modular surface: indeed, we will show that 
$$L=\{\limsup\limits_{t\to+\infty} H(g_t(x))<\infty: x \textrm{ is a unit cotangent vector to the modular surface}\}$$

\begin{remark} This fact is not surprising to experts: the Gauss map appears naturally by quotienting out the weak-stable manifolds of $g_t$ as observed by Artin, Series, Arnoux, ... (see, e.g., \cite{Ar}).
\end{remark}

An \emph{unimodular lattice} in $\mathbb{R}^2$ has the form $g(\mathbb{Z}^2)$, $g\in SL(2,\mathbb{Z})$, and the stabilizer in $SL(2,\mathbb{R})$ of the standard lattice $\mathbb{Z}^2$ is $SL(2,\mathbb{Z})$. In particular, the space of unimodular lattices in $\mathbb{R}^2$ is $SL(2,\mathbb{R})/SL(2,\mathbb{Z})$. 

As it turns out, $SL(2,\mathbb{R})/SL(2,\mathbb{Z})$ is the unit cotangent bundle to the \emph{modular surface} $\mathbb{H}/SL(2,\mathbb{Z})$ (where $\mathbb{H}=\{z\in\mathbb{C}: \textrm{Im}(z)>0\}$ is the hyperbolic upper-half plane and $\left(\begin{array}{cc} a & b \\ c & d \end{array}\right)\in SL(2,\mathbb{R})$ acts on $z\in\mathbb{H}$ via $\left(\begin{array}{cc} a & b \\ c & d \end{array}\right)\cdot z = \frac{az+b}{cz+d}$). 

The \emph{geodesic flow} of the modular surface is the action of $g_t=\left(\begin{array}{cc} e^t & 0 \\ 0 & e^{-t} \end{array}\right)$ on $SL(2,\mathbb{R})/SL(2,\mathbb{Z})$. The \emph{stable} and \emph{unstable manifolds} of $g_t$ are the orbits of the \emph{stable} and \emph{unstable horocycle flows} $h_s=\left(\begin{array}{cc} 1 & 0 \\ s & 1 \end{array}\right)$ and $u_s =\left(\begin{array}{cc} 1 & s \\ 0 & 1 \end{array}\right)$: indeed, this follows from the facts that $g_t h_s = h_{s e^{-2t}} g_t$ and $g_t u_s = u_{s e^t} g_t$. 

The set of \emph{holonomy} (or \emph{primitive}) \emph{vectors} of $\mathbb{Z}^2$ is 
$$\textrm{Hol}(\mathbb{Z}^2):=\{(p,q)\in\mathbb{Z}^2: \textrm{gcd}(p,q)=1\}$$
In general, the set $\textrm{Hol}(X)$ of holonomy vectors of $X=g(\mathbb{Z}^2)$, $g\in SL(2,\mathbb{Z})$, is 
$$\textrm{Hol}(X):=g(\textrm{Hol}(\mathbb{Z}^2))\subset \mathbb{R}^2$$

The \emph{systole} $\textrm{sys}(X)$ of $X=g(\mathbb{Z}^2)$ is 
$$\textrm{sys}(X) := \min\{\|v\|_{\mathbb{R}^2}: v\in\textrm{Hol}(X)\}$$ 

\begin{remark} By Mahler's compactness criterion \cite{Mahler}, $X\mapsto \frac{1}{\textrm{sys}(X)}$ is a proper function on $SL(2,\mathbb{R})/SL(2,\mathbb{Z})$. 
\end{remark}

\begin{remark} For later reference, we write $\textrm{Area}(v):=|\textrm{Re}(v)|\cdot |\textrm{Im}(v)|$ for the area of the rectangle in $\mathbb{R}^2$ with diagonal $v=(\textrm{Re}(v), \textrm{Im}(v))\in \mathbb{R}^2$. 
\end{remark}

\begin{proposition} The forward geodesic flow orbit of $X\in SL(2,\mathbb{R})/SL(2,\mathbb{Z})$ does not go straight to infinity (i.e., $\textrm{sys}(g_t(X))\to 0$ as $t\to+\infty$) if and only if there is no vertical vector in  $\textrm{Hol}(X)$. In this case, there are (unique) parameters $s, t, \alpha\in\mathbb{R}$ such that 
$$X = h_s g_t u_{-\alpha}(\mathbb{Z}^2)$$ 
\end{proposition}

\begin{proof} By unimodularity, any $X = g(\mathbb{Z}^2)$ has a single \emph{short} holonomy vector. Since $g_t$ contracts vertical vectors and expands horizontal vectors for $t>0$, we have that $\textrm{sys}(g_t(X))\to 0$ as $t\to+\infty$ if and only if $\textrm{Hol}(X)$ contains a vertical vector. 

By Iwasawa decomposition, there are (unique) parameters $s, t, \theta\in\mathbb{R}$ such that $X=h_s g_t r_{\theta}$, where $r_{\theta} = \left( \begin{array}{cc} \cos\theta & -\sin\theta \\ \sin\theta & \cos\theta \end{array} \right)$. Since $\cos\theta\neq 0$ when $\textrm{Hol}(X)$ contains no vertical vector and, in this situation, 
$$r_{\theta} = h_{\tan\theta} g_{\log\cos\theta} u_{-\tan\theta},$$
we see that $X=h_{s+e^{-2t}\tan\theta} \cdot g_{t+\log\cos\theta} \cdot u_{-\tan\theta}(\mathbb{Z}^2)$ 
(because $h_s g_t r_{\theta}= h_s g_t h_{\tan\theta} g_{\log\cos\theta} u_{-\tan\theta}=h_{s+e^{-2t}\tan\theta} \cdot g_{t+\log\cos\theta} \cdot u_{-\tan\theta}$). This ends the proof of the proposition. 
\end{proof}

\begin{proposition} Let $X=h_s g_t u_{-\alpha}(\mathbb{Z}^2)$ be an unimodular lattice without vertical holonomy vectors. Then, 
$$\ell(\alpha) = \limsup\limits_{\substack{|\textrm{Im}(v)|\to\infty \\ v\in\textrm{Hol}(X)}} \frac{1}{\textrm{Area}(v)} = \limsup\limits_{T\to+\infty} \frac{2}{\textrm{sys}(g_T(X))^2}$$
\end{proposition} 

\begin{remark} This proposition says that the dynamical quantity $\limsup\limits_{T\to+\infty} \frac{2}{\textrm{sys}(g_T(X))^2}$ does \emph{not} depend on the ``weak-stable part'' $h_s g_t$ (but only on $\alpha$) and it can be computed \emph{without} dynamics by simply studying almost vertical holonomy vectors in $X$. 
\end{remark} 

\begin{proof} Note that $\textrm{Area}(g_t(v)) = \textrm{Area}(v)$ for all $t\in\mathbb{R}$ and $v\in\mathbb{R}^2$. Since $\textrm{Area}(v) = \frac{\|g_{t(v)}(v)\|^2}{2}$ for $t(v):=\frac{1}{2}\log\frac{|\textrm{Im}(v)|}{|\textrm{Re}(v)|}$, the equality $\limsup\limits_{\substack{|\textrm{Im}(v)|\to\infty \\ v\in\textrm{Hol}(X)}} \frac{1}{\textrm{Area}(v)} = \limsup\limits_{T\to+\infty} \frac{2}{\textrm{sys}(g_T(X))^2}$ follows. 

The relation $g_T h_s = h_{s e^{-2T}} g_T$ and the continuity of the systole function imply that $\limsup\limits_{T\to+\infty} \frac{2}{\textrm{sys}(g_T(X))^2}$ depends only on $\alpha$. Because any $v\in\textrm{Hol}(u_{-\alpha}(\mathbb{Z}^2))$ has the form $v=(p-q\alpha, q) = u_{-\alpha}(p,q)$ with $(p,q)\in\textrm{Hol}(\mathbb{Z}^2)$, the equality $\limsup\limits_{\substack{|\textrm{Im}(v)|\to\infty \\ v\in\textrm{Hol}(X)}} \frac{1}{\textrm{Area}(v)} = \ell(\alpha)$. 
\end{proof}

In summary, the previous proposition says that the Lagrange spectrum $L$ coincides with  
$$\{\limsup\limits_{T\to+\infty} H(g_T(x))<\infty: x\in SL(2,\mathbb{R})/SL(2,\mathbb{Z})\}$$ 
where $H(y) = \frac{2}{\textrm{sys}(y)^2}$ is a (proper) height function and $g_t$ is the geodesic flow on $SL(2,\mathbb{R})/SL(2,\mathbb{Z})$. 
\begin{figure}[htb!]
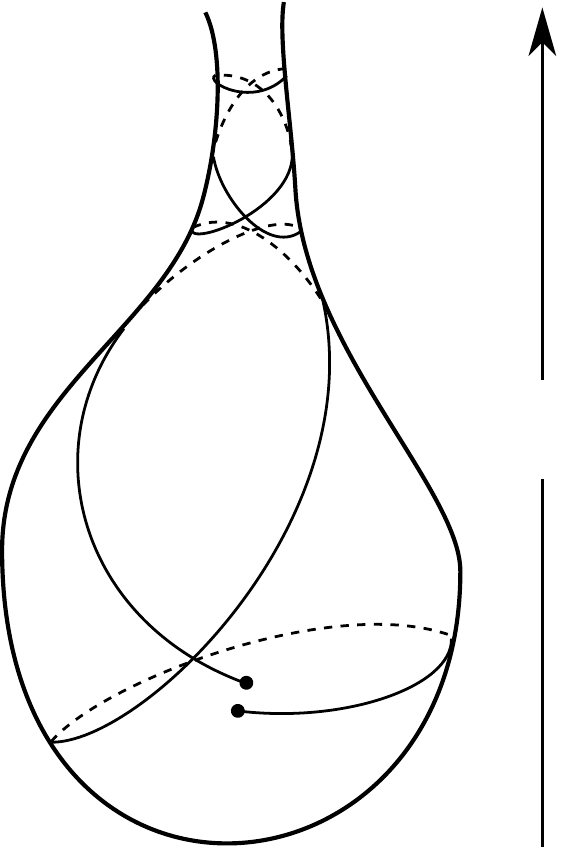
\end{figure}
\begin{remark} Several number-theoretical problems translate into dynamical questions on the modular surface: for example, Zagier \cite{Za-RH} showed that the Riemann hypothesis is equivalent to a certain speed of equidistribution of $u_s$-orbits on $SL(2,\mathbb{R})/SL(2,\mathbb{Z})$. 
\end{remark}

\subsection{Hall's ray and Freiman's constant}

In 1947, M. Hall \cite{Hall} proved that: 

\begin{theorem}[Hall]\label{t.Hall} The half-line $[6,+\infty)$ is contained in $L$. 
\end{theorem} 

This result motivates the following nomenclature: the biggest half-line $[c_F,+\infty)\subset L (\subset M)$ is called \emph{Hall's ray}. 

In 1975, G. Freiman \cite{Fr75} determined Hall's ray: 

\begin{theorem}[Freiman] $c_F = 4+\frac{253589820+283798\sqrt{462}}{491993569}\simeq 4.527829566...$
\end{theorem}

The constant $c_F$ is called \emph{Freiman's constant}. 

Let us sketch the proof of Hall's theorem based on the following lemma:

\begin{lemma}[Hall]\label{l.Hall} Denote by $C(4):=\{[0; a_1, a_2, \dots]\in\mathbb{R}: a_i\in\{1, 2, 3, 4\} \,\,\forall\, i\in\mathbb{N}\}$. Then, 
$$C(4)+C(4):=\{x+y\in\mathbb{R}: x, y\in C(4)\} = [\sqrt{2}-1, 4(\sqrt{2}-1)] = [0.414\dots, 1.656\dots]$$ 
\end{lemma}

\begin{remark} The reader can find a proof of this lemma in Cusick-Flahive's book \cite{CF}. Interestingly enough, some of the techniques in the proof of Hall's lemma were rediscovered much later (in 1979) in the context of Dynamical Systems by Newhouse \cite{Newhouse} (in the proof of his \emph{gap lemma}).  
\end{remark}

\begin{remark} $C(4)$ is a \emph{dynamical Cantor set}\footnote{See Subsections \ref{ss.dynamical-Cantor} and \ref{ss.Gauss-Cantor} below.} whose Hausdorff dimension is $>1/2$ (see Remark \ref{r.JP} below). In particular, $C(4)\times C(4)$ is a planar Cantor set of Hausdorff dimension $>1$ and Hall's lemma says that its image $f(C(4)\times C(4)) = C(4)+C(4)$ under the the projection $f(x,y) = x+y$ contains an interval. Hence, Hall's lemma can be thought as a sort of ``particular case'' of \emph{Marstrand's theorem} \cite{Marstrand} (ensuring that typical projections of planar sets with Hausdorff dimension $>1$ has positive Lebesgue measure).
\end{remark}

For our purposes, the specific form $C(4)+C(4)$ is \emph{not} important: the \emph{key point} is that $C(4)+C(4)$ is an interval of length $>1$. 

Indeed, given $6\leq\ell<\infty$, Hall's lemma guarantees the existence of $c_0\in\mathbb{N}$, $5\leq c_0\leq\ell$ such that $\ell-c_0\in C(4)+C(4)$. Thus, 
$$\ell = c_0 + [0; a_1, a_2,\dots] + [0; b_1, b_2,\dots]$$ 
with $a_i, b_i\in\{1,2,3,4\}$ for all $i\in\mathbb{N}$. 

Define 
$$\alpha:=[0; \underbrace{b_1, c_0, a_1}_{1^{st} \textrm{ block}}, \dots, \underbrace{b_n, \dots, b_1, c_0, a_1, \dots, a_n}_{n^{th} \textrm{ block}}, \dots]$$
Since $c_0\geq 5 > 4\geq a_i, b_i$ for all $i\in\mathbb{N}$, Perron's characterization of $\ell(\alpha)$ implies that 
$$L\ni \ell(\alpha) = \lim\limits_{n\to\infty} (c_0 + [0; a_1, a_2, \dots, a_n] + [0; b_1, b_2, \dots, b_n]) = \ell$$ 
This proves Theorem \ref{t.Hall}. 

\subsection{Statement of Moreira's theorem}

Our discussion so far can be summarized as follows: 
\begin{itemize}
\item $L\cap (-\infty, 3) = M\cap (-\infty, 3) = \{k_1<k_2<\dots<k_n<\dots\}$ is an \emph{explicit} discrete set; 
\item $L\cap [c_F,\infty) = M\cap [c_F,\infty)$ is an \emph{explicit} ray. 
\end{itemize}

Moreira's theorem \cite{Moreira} says that the \emph{intermediate parts} $L\cap [3, c_F]$ and $M\cap [3, c_F]$ of the Lagrange and Markov spectra have an intricate structure: 

\begin{theorem}[Moreira]\label{t.Gugu} For each $t\in\mathbb{R}$, the sets $L\cap (-\infty, t)$ and $M\cap (-\infty, t)$ have the same Hausdorff dimension, say $d(t)\in [0,1]$. 

Moreover, the function $t\mapsto d(t)$ is continuous, $d(3+\varepsilon)>0$ for all $\varepsilon>0$ and $d(\sqrt{12})=1$ (even though  $\sqrt{12}=3.4641... < 4.5278...=c_F$). 
\end{theorem} 

\begin{remark} Many results about $L$ and $M$ are \emph{dynamical}\footnote{I.e., they involve Perron's characterization of $L$ and $M$, the study of Gauss map and/or the geodesic flow on the modular surface, etc.}. In particular, it is not surprising that many facts about $L$ and $M$ have counterparts for \emph{dynamical Lagrange and Markov spectra}\footnote{I.e., the collections of ``records'' of height functions along orbits of dynamical systems.}: for example, Hall ray or intervals in dynamical Lagrange spectra were found by Parkkonen-Paulin \cite{PP}, Hubert-Marchese-Ulcigrai \cite{HMU} and Moreira-Roma\~na \cite{MR}, and the continuity result in Moreira's theorem \ref{t.Gugu} was recently extended by Cerqueira, Moreira and the author in \cite{CMM}. 
\end{remark} 

Before entering into the proof of Moreira's theorem, let us close this section by briefly recalling the notion of Hausdorff dimension. 

\subsection{Hausdorff dimension} 

The $s$-\emph{Hausdorff measure} $m_s(X)$ of a subset $X\subset\mathbb{R}^n$ is 
$$m_s(X):=\lim\limits_{\delta\to 0}\inf\limits_{\substack{\bigcup\limits_{i\in\mathbb{N}} U_i \supset X, \\ \textrm{diam}(U_i)\leq \delta \,\, \forall\, i\in\mathbb{N}}} \sum\limits_{i\in\mathbb{N}} \textrm{diam}(U_i)^s$$  

The \emph{Hausdorff dimension} of $X$ is 
$$HD(X):=\sup\{s\in\mathbb{R}: m_s(X)=\infty\} = \inf\{s\in\mathbb{R}: m_s(X)=0\}$$

\begin{remark}\label{r.box-counting} There are many notions of dimension in the literature: for example, the \emph{box-counting dimension} of $X$ is $\lim\limits_{\delta\to 0}\frac{\log N_X(\delta)}{\log(1/\delta)}$ where $N_X(\delta)$ is the smallest number of boxes of side lengths $\leq\delta$ needed to cover $X$. As an exercise, the reader is invited to show that the Hausdorff dimension is always smaller than or equal to the box-counting dimension. 
\end{remark} 

The following exercise (whose solution can be found in Falconer's book \cite{Falconer}) describes several elementary properties of the Hausdorff dimension: 

\begin{exercise}\label{exercise.HD} Show that: 
\begin{itemize}
\item[(a)] if $X\subset Y$, then $HD(X)\leq HD(Y)$; 
\item[(b)] $HD(\bigcup\limits_{i\in\mathbb{N}} X_i) = \sup\limits_{i\in\mathbb{N}} HD(X_i)$; in particular, $HD(X)=0$ whenever $X$ is a countable set (such as $X=\{p\}$ or $X=\mathbb{Q}^n$); 
\item[(c)] if $f:X\to Y$ is $\alpha$-H\"older continuous\footnote{I.e., for some constant $C>0$, one has $|f(x)-f(x')|\leq C |x-x'|^{\alpha}$ for all $x, x'\in X$.}, then $\alpha\cdot HD(f(X))\leq HD(X)$; 
\item[(d)] $HD(\mathbb{R}^n) = n$ and, more generally, $HD(X)=m$ when $X\subset\mathbb{R}^n$ is a smooth $m$-dimensional submanifold. 
\end{itemize}
\end{exercise} 

\begin{example} Cantor's middle-third set $C=\{\sum\limits_{i=1}^{\infty}\frac{a_i}{3^i}: a_i\in\{0,2\}\,\, \forall \, i\in\mathbb{N}\}$ has Hausdorff dimension $\frac{\log 2}{\log 3}\in (0,1)$: see Falconer's book \cite{Falconer} for more details. 
\end{example} 

Using item (c) of Exercise \ref{exercise.HD} above, we have the following corollary of Moreira's theorem \ref{t.Gugu}: 

\begin{corollary}[Moreira] The function $t\mapsto HD(L\cap(-\infty,t))$ is not $\alpha$-H\"older continuous for any $\alpha>0$. 
\end{corollary} 

\begin{proof} By Theorem \ref{t.Gugu}, $d$ maps $L\cap[3,3+\varepsilon]$ to the non-trivial interval $[0, d(3+\varepsilon)]$ for any  $\varepsilon>0$. By item (c) of Exercise \ref{exercise.HD}, if $t\mapsto d(t)=HD(L\cap(-\infty, t))$ were $\alpha$-H\"older continuous for some $\alpha>0$, then it would follow that 
$$0<\alpha=\alpha\cdot HD([0,d(3+\varepsilon)])\leq HD(L\cap [3,3+\varepsilon]) = d(3+\varepsilon)$$
for all $\varepsilon>0$. On the other hand, Theorem \ref{t.Gugu} (and item (b) of Exercise \ref{exercise.HD})  also says that 
$$\lim\limits_{\varepsilon\to 0} d(3+\varepsilon)=d(3)=HD(L\cap(-\infty,3))=0$$
In summary, $0<\alpha\leq\lim\limits_{\varepsilon\to 0}d(3+\varepsilon) = 0$, a contradiction.  
\end{proof}

%%%%%%%%%%%%%%%%%%%%%%%%%%%%%%%%%%%
%%%%%%%%%%%%%%%%%%%%%%%%%%%%%%%%%%%
%%%%%%%%%%%%%%%%%%%%%%%%%%%%%%%%%%%
%%%%%%%%%%%%%%%%%%%%%%%%%%%%%%%%%%%
%%%%%%%%%%%% Section 2 %%%%%%%%%%%%%%%%
%%%%%%%% Proof of Moreira's theorem %%%%%%%%%
%%%%%%%%%%%%%%%%%%%%%%%%%%%%%%%%%%%
%%%%%%%%%%%%%%%%%%%%%%%%%%%%%%%%%%%
%%%%%%%%%%%%%%%%%%%%%%%%%%%%%%%%%%%
%%%%%%%%%%%%%%%%%%%%%%%%%%%%%%%%%%%

\section{Proof of Moreira's theorem}

\subsection{Strategy of proof of Moreira's theorem} Roughly speaking, the continuity of $d(t)=HD(L\cap(-\infty, t))$ is proved in four steps: 

\begin{itemize}
\item if $0<d(t)<1$, then for all $\eta>0$ there exists $\delta>0$ such that $L\cap(-\infty, t-\delta)$ can be ``\emph{approximated from inside}'' by $K+K'=f(K\times K')$ where $K$ and $K'$ are \emph{Gauss-Cantor sets} with $HD(K)+HD(K')=HD(K\times K')>(1-\eta)d(t)$ (and $f(x,y)=x+y$);  
\item by \emph{Moreira's dimension formula} (derived from profound works of Moreira and Yoccoz on the geometry of Cantor sets), we have that 
$$HD(f(K\times K')) = HD(K\times K')$$
\item thus, if $0<d(t)<1$, then for all $\eta>0$ there exists $\delta>0$ such that 
$$d(t-\delta)\geq HD(f(K\times K')) = HD(K\times K')\geq (1-\eta)d(t);$$
hence, $d(t)$ is \emph{lower semicontinuous}; 
\item finally, an elementary compactness argument shows the \emph{upper semicontinuity} of $d(t)$.
\end{itemize}

\begin{remark} This strategy is \emph{purely dynamical} because the particular forms of the height function $f$ and the Gauss map $G$ are \emph{not} used. Instead, we just need the \emph{transversality} of the gradient of $f$ to the stable and unstable manifolds (vertical and horizontal axis) and the \emph{non-essential affinity} of Gauss-Cantor sets. (See \cite{CMM} for more explanations.)
\end{remark}

In the remainder of this section, we will implement (a version of) this strategy in order to deduce the continuity result in Theorem \ref{t.Gugu}.  

\subsection{Dynamical Cantor sets}\label{ss.dynamical-Cantor} A \emph{dynamically defined Cantor set} $K\subset \mathbb{R}$ is 
$$K=\bigcap\limits_{n\in\mathbb{N}}\psi^{-n}(I_1\cup\dots\cup I_k)$$
where $I_1,\dots, I_k$ are pairwise disjoint compact intervals, and $\psi:I_1\cup\dots\cup I_k\to I$ is a $C^r$-map from $I_1\cup\dots\cup I_k$ to its convex hull $I$ such that:
\begin{itemize}
\item $\psi$ is \emph{uniformly expanding}: $|\psi'(x)|>1$ for all $x\in I_1\cup\dots\cup I_k$; 
\item $\psi$ is a (full) \emph{Markov map}: $\psi(I_j)=I$ for all $1\leq j\leq k$.
\end{itemize}

\begin{remark} Dynamical Cantor sets are usually defined with a weaker Markov condition, but we stick to this definition for simplicity. 
\end{remark}

\begin{example} Cantor's middle-third set $C=\{\sum\limits_{i=1}^{\infty} \frac{a_i}{3^i}: a_i\in\{0,2\} \,\, \forall\,i\in\mathbb{N}\}$ is 
$$C=\bigcap\limits_{n\in\mathbb{N}} \psi^{-n}([0,1/3]\cup[2/3,1])$$
where $\psi:[0, 1/3]\cup [2/3, 1]\to [0,1]$ is given by
$$\psi(x)=\left\{\begin{array}{cl} 3x, & \textrm{if } 0\leq x\leq 1/3 \\ 3x-2, & \textrm{if } 2/3\leq x\leq 1\end{array}\right.$$
\begin{figure}[htb!]
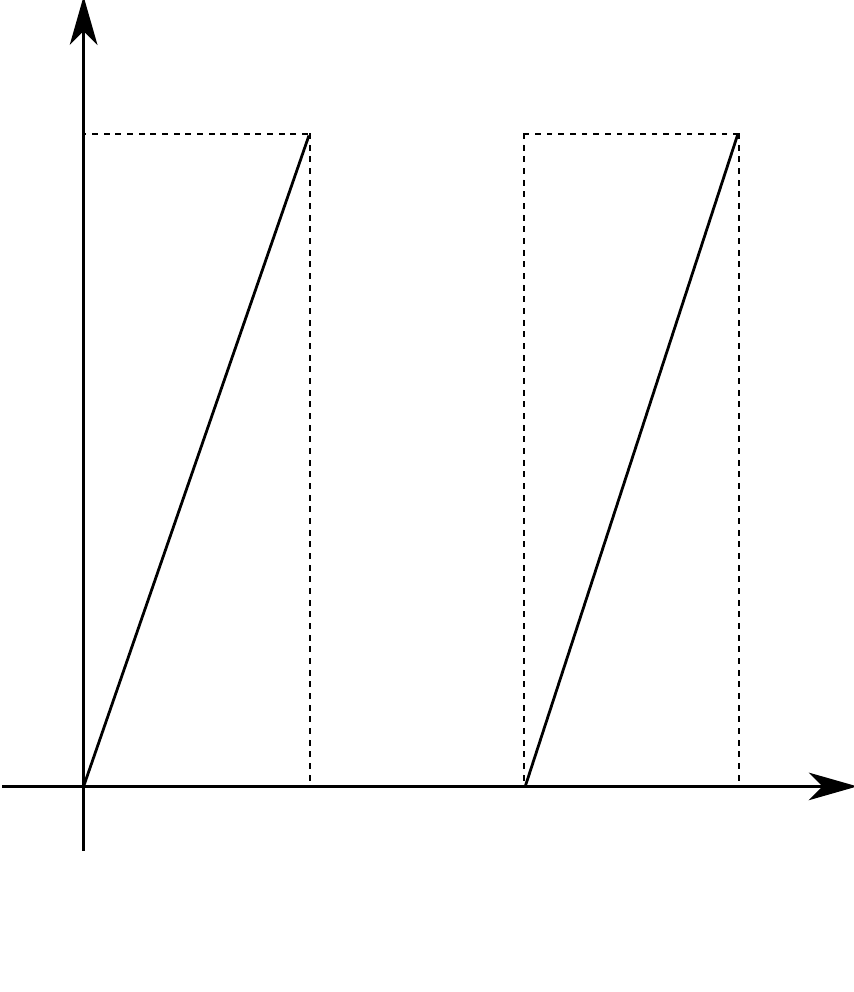
\end{figure}
\end{example}

\begin{remark} A dynamical Cantor set is called \emph{affine} when $\psi|_{I_j}$ is affine for all $j$. In this language, Cantor's middle-third set is an \emph{affine dynamical Cantor set}. 
\end{remark}

\begin{example} Given $A\geq 2$, let $C(A):=\{[0;a_1, a_2,\dots]: 1\leq a_i\leq A\,\,\forall\,i\in\mathbb{N}\}$. This is a dynamical Cantor set associated to Gauss map: for example, 
$$C(2) = \bigcap\limits_{n\in\mathbb{N}} G^{-n}(I_1\cup I_2)$$
where $I_1$ and $I_2$ are the intervals depicted below. 
\begin{figure}[htb!]\label{f.C(2)}
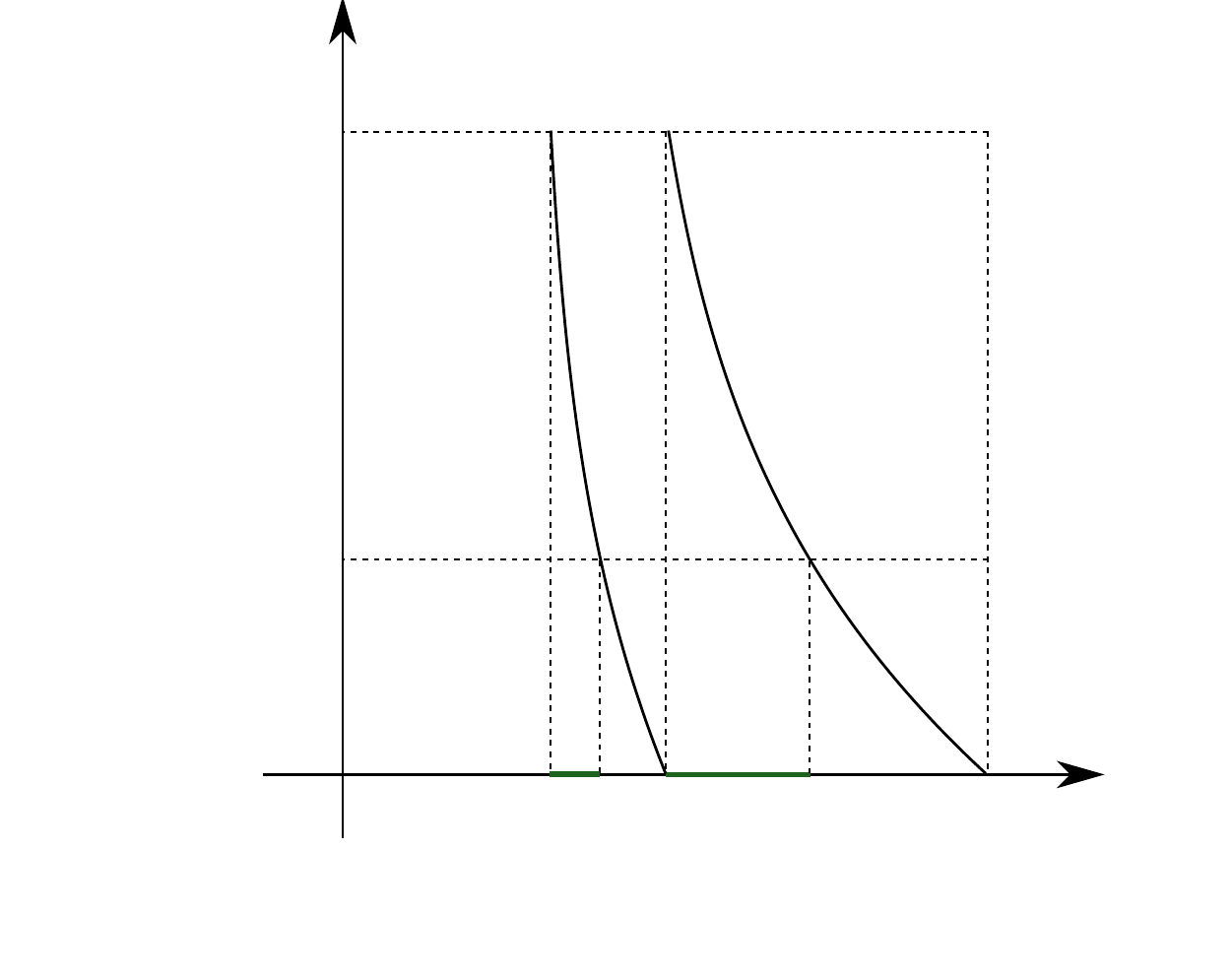
\end{figure}
\end{example}

\begin{remark}\label{r.JP} Hensley \cite{He} showed that 
$$HD(C(A)) = 1-\frac{6}{\pi^2 A} - \frac{72\log A}{\pi^4 A^2} + O(\frac{1}{A^2}) = 1-\frac{1+o(1)}{\zeta(2)A}$$ 
and Jenkinson-Pollicott \cite{JePo1}, \cite{JePo2} used thermodynamical formalism methods to obtain that 
$$HD(C(2)) = 0.53128050627720514162446864736847178549305910901839\dots,$$
$$HD(C(3)) \simeq 0.705\dots, \quad HD(C(4)) \simeq 0.788\dots$$
\end{remark}

\subsection{Gauss-Cantor sets}\label{ss.Gauss-Cantor} 

The set $C(A)$ above is a particular case of \emph{Gauss-Cantor set}:

\begin{definition} Given $B=\{\beta_1,\dots,\beta_l\}$, $l\geq 2$, a finite, primitive\footnote{I.e., $\beta_i$ doesn't begin by $\beta_j$ for all $i\neq j$.} alphabet of finite words $\beta_j\in(\mathbb{N}^*)^{r_j}$, the Gauss-Cantor set $K(B)\subset [0,1]$ associated to $B$ is 
$$K(B):=\{[0;\gamma_1, \gamma_2, \dots]: \gamma_i\in B\,\,\forall\, i\}$$ 
\end{definition}

\begin{example} $C(A) = K(\{1,\dots, A\})$. 
\end{example}

\begin{exercise}\label{ex.Gauss-Cantor} Show that any Gauss-Cantor set $K(B)$ is dynamically defined.\footnote{Hint: For each word $\beta_j\in(\mathbb{N}^*)^{r_j}$, let $I(\beta_j)=\{[0;\beta_j, a_1,\dots]:a_i\in\mathbb{N}\,\,\forall\,i\}=I_j$ and $\psi|_{I_j}:=G^{r_j}$ where $G(x)=\{1/x\}$ is the Gauss map.}
\end{exercise}

From the \emph{symbolic} point of view, $B=\{\beta_1,\dots,\beta_l\}$ as above induces a subshift $$\Sigma(B)=\{(\gamma_i)_{i\in\mathbb{Z}}:\gamma_i\in B \,\,\forall\,i\}\subset \Sigma=(\mathbb{N}^*)^{\mathbb{Z}} = \Sigma^-\times\Sigma^+:=(\mathbb{N}^*)^{\mathbb{Z}^-}\times (\mathbb{N}^*)^{\mathbb{N}}$$
Also, the corresponding Gauss-Cantor is $K(B)=\{[0;\gamma]:\gamma\in\Sigma^+(B)\}$ where $\Sigma^+(B) = \pi^+(\Sigma(B))$ and $\pi^+:\Sigma\to\Sigma^+$ is the natural projection (related to local unstable manifolds of the left shift map on $\Sigma$). 

For later use, denote by $B^T=\{\beta^T:\beta\in B\}$ the \emph{transpose} of $B$, where $\beta^T:=(a_n,\dots, a_1)$ for $\beta=(a_1,\dots,a_n)$. 

The following proposition (due to Euler) is proved in Appendix \ref{a.Euler}: 
\begin{proposition}[Euler]\label{p.Euler} If $[0;\beta]=\frac{p_n}{q_n}$, then $[0;\beta^T]=\frac{r_n}{q_n}$.
\end{proposition}

A striking consequence of this proposition is: 
\begin{corollary}\label{c.Euler} $HD(K(B)) = HD(K(B^T))$.
\end{corollary}

\begin{proof}[Sketch of proof] The lengths of the intervals $I(\beta) = \{[0;\beta, a_1,\dots]: a_i\in\mathbb{N} \,\,\forall\,i \}$ in the construction of $K(B)$ depend only on the denominators of the partial quotients of $[0;\beta]$. Therefore, we have from Proposition \ref{p.Euler} that $K(B)$ and $K(B^T)$ are Cantor sets constructed from intervals with same lengths, and, \emph{a fortiori}, they have the Hausdorff dimension. 
\end{proof}

\begin{remark} This corollary is closely related to the existence of \emph{area-preserving} natural extensions of Gauss map (see \cite{Ar}) and the coincidence of stable and unstable dimensions of a horseshoe of an area-preserving surface diffeomorphism (see \cite{McCMa}).  
\end{remark}

\subsection{Non-essentially affine Cantor sets} We say that 
$$K=\bigcap\limits_{n\in\mathbb{N}} \psi^{-n}(I_1\cup\dots\cup I_r)$$ 
is \emph{non-essentially affine} if there is \emph{no} global conjugation $h\circ\psi\circ h^{-1}$ such that \emph{all} branches 
$$(h\circ \psi\circ h^{-1})|_{h(I_j)}, \,\,\, j=1, \dots, r$$ 
are affine maps of the real line. 

Equivalently, if $p\in K$ is a periodic point of $\psi$ of period $k$ and $h:I\to I$ is a diffeomorphism of the convex hull $I$ of $I_1\cup\dots \cup I_r$ such that $h\circ\psi^k\circ h^{-1}$ is affine\footnote{Such a diffeomorphism $h$ linearizing \emph{one} branch of $\psi$ always exists by Poincar\'e's linearization theorem.} on $h(J)$ where $J$ is the connected component of the domain of $\psi^k$ containing $p$, then $K$ is non-essentially affine if and only if $(h\circ\psi\circ h^{-1})''(x)\neq 0$ for some $x\in h(K)$. 

\begin{proposition}\label{p.non-ess-aff-Gauss} Gauss-Cantor sets are non-essentially affine.
\end{proposition} 

\begin{proof} The basic idea is to explore the fact that the second derivative of a non-affine M\"obius transformation never vanishes. 

More concretely, let $B=\{\beta_1,\dots,\beta_m\}$, $\beta_j\in(\mathbb{N}^*)^{r_j}$, $1\leq j\leq m$. For each $\beta_j$, let 
$$x_j:=[0;\beta_j,\beta_j,\dots]\in I_j=I(\beta_j)\subset\{[0;\beta_j,\alpha]:\alpha\geq 1\}$$ 
be the fixed point of the branch $\psi|_{I_j}=G^{r_j}$ of the expanding map $\psi$ naturally\footnote{Cf. Exercise \ref{ex.Gauss-Cantor}.} defining the Gauss-Cantor set $K(B)$.

By Corollary \ref{c.Mobius-fraction}, $\psi|_{I_j}(x)=\frac{q^{(j)}_{r_{j}-1}x - p^{(j)}_{r_{j}-1}}{p^{(j)}_{r_j} - q^{(j)}_{r_j}x}$ where $\frac{p^{(j)}_k}{q^{(j)}_k} = [0; b^{(j)}_1, \dots, b^{(j)}_k]$ and $\beta_j = (b^{(j)}_1, \dots, b^{(j)}_{r_j})$. 

Note that the fixed point $x_j$ of $\psi|_{I_j}$ is the positive solution of the second degree equation 
$$q^{(j)}_{r_j} x^2 + (q^{(j)}_{r_{j}-1} - p^{(j)}_{r_j}) x - p^{(j)}_{r_{j}-1} = 0$$
In particular, $x_j$ is a \emph{quadratic surd}. 

For each $1\leq j\leq k$, the M\"obius transformation $\psi|_{I_j}$ has a hyperbolic fixed point $x_j$. It follows (from Poincar\'e linearization theorem) that there exists a M\"obius transformation 
$$\alpha_j(x) = \frac{a_j x + b_j}{c_j x + d_j}$$ 
linearizing $\psi|_{I_j}$, i.e., $\alpha_j(x_j)=x_j$, $\alpha'(x_j)=1$ and $\alpha_j\circ (\psi|_{I_j})\circ \alpha_j^{-1}$ is an affine map. 

Since non-affine M\"obius transformations have non-vanishing second derivative, the proof of the proposition will be complete once we show that $\alpha_1\circ(\psi|_{I_2})\circ\alpha_1^{-1}$ is not affine. So, let us suppose by contradiction that $\alpha_1\circ(\psi|_{I_2})\circ\alpha_1^{-1}$ is affine. In this case, $\infty$ is a common fixed point of the (affine) maps $\alpha_1\circ(\psi|_{I_2})\circ\alpha_1^{-1}$ and $\alpha_1\circ(\psi|_{I_1})\circ\alpha_1^{-1}$, and, \emph{a fortiori}, $\alpha_1^{-1}(\infty) = -d_1/c_1$ is a common fixed point of $\psi|_{I_1}$ and $\psi|_{I_2}$. Thus, the second degree equations 
$$q^{(1)}_{r_1} x^2 + (q^{(1)}_{r_1-1} - p^{(1)}_{r_1}) x - p^{(j)}_{r_1-1} = 0 \quad \textrm{and} \quad q^{(2)}_{r_2} x^2 + (q^{(2)}_{r_2-1} - p^{(2)}_{r_2}) x - p^{(2)}_{r_2-1} = 0$$ 
would have a common root. This implies that these polynomials coincide (because they are polynomials in $\mathbb{Z}[x]$ which are irreducible\footnote{Thanks to the fact that their roots $x_1, x_2\notin\mathbb{Q}$.}) and, hence, their other roots $x_1$, $x_2$ must coincide, a contradiction. 
\end{proof}

\subsection{Moreira's dimension formula} The Hausdorff dimension of projections of products of non-essentially affine Cantor sets is given by the following formula:

\begin{theorem}[Moreira]\label{t.Moreira-dim} Let $K$ and $K'$ be two $C^2$ dynamical Cantor sets. If $K$ is non-essentially affine, then the projection $f(K\times K')=K+K'$ of $K\times K'$ under $f(x,y)=x+y$ has Hausdorff dimension 
$$HD(f(K+K')) = \min\{1, HD(K)+HD(K')\}$$
\end{theorem}

\begin{remark} This statement is a \emph{particular} case of Moreira's dimension formula (which is sufficient for our current purposes because Gauss-Cantor sets are non-essentially affine). 
\end{remark}

The proof of this result is out of the scope of these notes: indeed, it depends on the techniques introduced in two works (from 2001 and 2010) by Moreira and Yoccoz \cite{MY01}, \cite{MY10} such as fine analysis of \emph{limit geometries} and \emph{renormalization operators}, ``recurrence on scales'', ``compact recurrent sets of relative configurations'', and \emph{Marstrand's theorem}. We refer the reader to \cite{Moreira-dim} for more details. 

\begin{remark} Moreira's dimension formula is coherent with Hall's Lemma \ref{l.Hall}: in fact, since $HD(C(4))>1/2$, it is natural that $HD(C(4)+C(4))=1$. 
\end{remark}

\subsection{First step towards Moreira's theorem \ref{t.Gugu}: projections of Gauss-Cantor sets}

Let $\Sigma(B)\subset(\mathbb{N}^*)^{\mathbb{Z}}$ be a complete shift of finite type. Denote by $\ell(\Sigma(B))$, resp. $m(\Sigma(B))$, the pieces of the Lagrange, resp. Markov, spectrum generated by $\Sigma(B)$, i.e., 
$$\ell(\Sigma(B)) = \{\ell(\underline{\theta}): \underline{\theta}\in\Sigma(B)\}, \,\, \textrm{resp.} \quad 
m(\Sigma(B)) = \{m(\underline{\theta}): \underline{\theta}\in\Sigma(B)\}$$ 
where $\ell(\underline{\theta})=\limsup\limits_{n\to\infty}f(\sigma^n(\underline{\theta}))$, $m(\underline{\theta})=\sup\limits_{n\in\mathbb{Z}}f(\sigma^n(\underline{\theta}))$, $f((\theta_i)_{i\in\mathbb{Z}}) = [\theta_0;\theta_1,\dots]+[0;\theta_{-1},\dots]$ and $\sigma((\theta_i)_{i\in\mathbb{Z}}) = (\theta_{i+1})_{i\in\mathbb{Z}}$ is the shift map. 

The following proposition relates the Hausdorff dimensions of the pieces of the Langrange and Markov spectra associated to $\Sigma(B)$ and the projection $f(K(B)\times K(B^T))$:

\begin{proposition}\label{p.Lagrange-Cantor} One has $HD(\ell(\Sigma(B))) = HD(m(\Sigma(B))) = \min\{1, 2\cdot HD(K(B))\}$.
\end{proposition}

\begin{proof}[Sketch of proof] By definition, 
$$\ell(\Sigma(B))\subset m(\Sigma(B))\subset \bigcup\limits_{a=1}^R(a+K(B)+K(B^T))$$
where $R\in\mathbb{N}$ is the largest entry among all words of $B$. 

Thus, $HD(\ell(\Sigma(B))) \leq HD(m(\Sigma(B))) \leq HD(K(B)) + HD(K(B^T))$. By Corollary \ref{c.Euler}, it follows that 
$$HD(\ell(\Sigma(B))) \leq HD(m(\Sigma(B))) \leq \min\{1, 2\cdot HD(K(B))\}$$ 

By Moreira's dimension formula (cf. Theorem \ref{t.Moreira-dim}), our task is now reduced to show that for all $\varepsilon>0$, there are ``replicas'' $K$ and $K'$ of Gauss-Cantor sets such that 
$$HD(K), HD(K') > HD(K(B))-\varepsilon \quad \textrm{and} \quad f(K\times K')=K+K'\subset \ell(\Sigma(B))$$ 

In this direction, let us order $B$ and $B^T$ by declaring that $\gamma<\gamma'$ if and only if $[0;\gamma] < [0;\gamma']$. 

Given $\varepsilon>0$, we can replace if necessary $B$ and/or $B^T$ by $B^n=\{\gamma_1\dots\gamma_n: \gamma_i\in B\,\,\forall\,i\}$ and/or $(B^T)^n$ for some large $n=n(\varepsilon)\in\mathbb{N}$ in such a way that 
$$HD(K(B^*)), HD(K((B^T)^*)) > HD(K(B))-\varepsilon$$
where $A^*:=\{\min A, \max A\}$. Indeed, this holds because the Hausdorff dimension of a Gauss-Cantor set $K(A)$ associated to an alphabet $A$ with a large number of words does not decrease too much after removing only two words from $A$.  

We \emph{expect} the values of $\ell$ on $((B^T)^*)^{\mathbb{Z}^-}\times (B^*)^{\mathbb{N}}$ to \emph{decrease} because we removed the minimal and maximal elements of $B$ and $B^T$ (and, in general, $[a_0; a_1, a_2, \dots]<[b_0; b_1, b_2, \dots]$ if and only if $(-1)^k(a_k-b_k)<0$ where $k$ is the smallest integer with $a_k\neq b_k$). 

In particular, this gives \emph{some} control on the values of $\ell$ on $((B^T)^*)^{\mathbb{Z}^-}\times (B^*)^{\mathbb{N}}$, but this does \emph{not} mean that $K(B^*)+K((B^T)^*)\subset\ell(\Sigma(B))$. 

We overcome this problem by studying \emph{replicas} of $K(B^*)$ and $K((B^T)^*)$. More precisely, let  $\widetilde{\theta} = (\dots,\widetilde{\gamma}_0,\widetilde{\gamma}_1,\dots)\in\Sigma(B)$, $\widetilde{\gamma}_i\in B$ for all $i\in\mathbb{Z}$, such that 
$$m(\widetilde{\theta}) = \max m(\Sigma(B))$$ 
is attained at a position in the block $\widetilde{\gamma}_0$. 

By compactness, there exists $\eta>0$ and $m\in\mathbb{N}$ such that any 
$$\theta=(\dots,\gamma_{-m-2},\gamma_{-m-1},\widetilde{\gamma}_{-m},\dots,\widetilde{\gamma}_0,\dots,\widetilde{\gamma}_m,\gamma_{m+1},\gamma_{m+2},\dots)$$
with $\gamma_i\in B^*$ for all $i>m$ and $\gamma_i\in (B^T)^*$ for all $i<-m$ satisfies: 
\begin{itemize}
\item $m(\theta)$ is attained in a position in the \emph{central block} $(\widetilde{\gamma}_{-m}, \dots, \widetilde{\gamma}_0, \dots, \widetilde{\gamma}_m)$; 
\item $f(\sigma^n(\theta)) < m(\theta)-\eta$ for any \emph{non-central position} $n$. 
\end{itemize}

By exploring these properties, it is possible to enlarge the central block to get a word called $\tau^{\#}=(a_{-N_1},\dots, a_0,\dots, a_{N_2})$ in Moreira's paper \cite{Moreira} such that the replicas 
$$K=\{[a_0; a_1,\dots, a_{N_2}, \gamma_1, \gamma_2,\dots]: \gamma_i\in B^*\,\,\forall\,i>0\}$$ 
and 
$$K'=\{[0; a_{-1},\dots, a_{-N_1}, \gamma_{-1}, \gamma_{-2},\dots]: \gamma_i\in (B^T)^*\,\,\forall\,i<0\}$$ 
of $K(B^*)$ and $K((B^T)^*)$ have the desired properties that 
$$K+K' = f(K\times K')\subset \ell(\Sigma(B))$$ 
and 
$$HD(K)=HD(K(B^*))>HD(K)-\varepsilon, \quad HD(K')=HD(K((B^T)^*))>HD(K(B^T))-\varepsilon$$ 
This completes our sketch of proof of the proposition. 
\end{proof}

\subsection{Second step towards Moreira's theorem \ref{t.Gugu}: upper semi-continuity} 

Let $\Sigma_t:=\{\theta\in (\mathbb{N}^*)^{\mathbb{Z}}: m(\theta)\leq t\}$ for $3\leq t < 5$. 

Our long term goal is to compare $\Sigma_t$ with its projection $K_t^+:=\{[0;\gamma]:\gamma\in\pi^+(\Sigma_t)\}$ on the unstable part (where $\pi^+:(\mathbb{N}^*)^{\mathbb{Z}}\to (\mathbb{N}^*)^{\mathbb{N}}$ is the natural projection). 

Given $\alpha=(a_1,\dots,a_n)$, its \emph{unstable scale} $r^+(\alpha)$ is 
$$r^+(\alpha) = \lfloor\log 1/(\textrm{length of }I^+(\alpha))\rfloor$$
where $I^+(\alpha)$ is the interval with extremities $[0;a_1,\dots,a_n]$ and $[0; a_1,\dots,a_n+1]$. 

Denote by 
$$P_r^+:=\{\alpha=(a_1,\dots, a_n): r^+(\alpha)\geq r, r^+(a_1,\dots,a_{n-1})<r\}$$ 
and 
$$C^+(t,r):=\{\alpha\in P_r^+: I^+(\alpha)\cap K_t^+\neq\emptyset\}.$$

\begin{remark} By symmetry (i.e., replacing $\gamma$'s by $\gamma^T$'s), we can define $K^-_t$, $r^-(\alpha)$, etc. 
\end{remark}

For later use, we observe that the unstable scales have the following behaviour under concatenations of words: 

\begin{exercise}\label{ex.subadditive} Show that $r^+(\alpha\beta k)\geq r^+(\alpha)+r^+(\beta)$ for all $\alpha$, $\beta$ finite words and for all $k\in\{1,2,3,4\}$.  
\end{exercise}

In particular, since the family of intervals 
$$\{I^+(\alpha\beta k): \alpha\in C^+(t,r), \beta\in C^+(t,s), 1\leq k\leq 4\}$$ 
covers $K_t^+$, it follows from Exercise \ref{ex.subadditive} that 
$$\# C^+(t,r+s)\leq 4\#C^+(t,r)\#C^+(t,s)$$ 
for all $r, s\in\mathbb{N}$ and, hence, the sequence $(4\#C^+(t,r))_{r\in\mathbb{N}}$ is \emph{submultiplicative}.

So, the \emph{box-counting dimension} (cf. Remark \ref{r.box-counting}) $\Delta^+(t)$ of $K_t^+$ is 
$$\Delta^+(t) = \inf\limits_{m\in\mathbb{N}}\frac{1}{m}\log(4\#C^+(t,m)) = 
\lim\limits_{m\to\infty} \frac{1}{m}\log\#C^+(t,m)$$

An elementary compactness argument shows that the upper-semicontinuity of $\Delta^+(t)$:

\begin{proposition}\label{p.upper-sc} The function $t\mapsto \Delta^+(t)$ is upper-semicontinuous. 
\end{proposition}

\begin{proof} For the sake of contradiction, assume that there exist $\eta>0$ and $t_0$ such that $\Delta^+(t)>\Delta^+(t_0)+\eta$ for all $t>t_0$. 

By definition, this means that there exists $r_0\in\mathbb{N}$ such that 
$$\frac{1}{r}\log\#C^+(t,r) > \Delta^+(t_0)+\eta$$ 
for all $r\geq r_0$ and $t>t_0$. 

On the other hand, $C^+(t,r)\subset C^+(s,r)$ for all $t\leq s$ and, by compactness, $C^+(t_0,r)=\bigcap\limits_{t>t_0} C^+(t,r)$. Thus, if $r\to\infty$ and $t\to t_0$, the inequality of the previous paragraph would imply that 
$$\Delta^+(t_0) > \Delta^+(t_0)+\eta,$$ 
a contradiction. 
\end{proof}

\subsection{Third step towards Moreira's theorem \ref{t.Gugu}: lower semi-continuity}

The main result of this subsection is the following theorem allowing us to ``approximate from inside'' $\Sigma_t$ by Gauss-Cantor sets. 

\begin{theorem}\label{t.lower-sc} Given $\eta>0$ and $3\leq t<5$ with $d(t):=HD(L\cap (-\infty,t))>0$, we can find $\delta>0$ and a Gauss-Cantor set $K(B)$ associated to $\Sigma(B)\subset\{1,2,3,4\}^{\mathbb{Z}}$ such that 
$$\Sigma(B)\subset \Sigma_{t-\delta} \quad \textrm{and} \quad HD(K(B))\geq (1-\eta)\Delta^+(t)$$
\end{theorem}

This theorem allows us to derive the continuity statement in Moreira's theorem \ref{t.Gugu}:
\begin{corollary}\label{c.Gugu-continuity} $\Delta^-(t)=\Delta^+(t)$ is a continuous function of $t$ and $d(t)=\min\{1,2\cdot\Delta^+(t)\}$. 
\end{corollary} 

\begin{proof} By Corollary \ref{c.Euler} and Theorem \ref{t.lower-sc}, we have that 
$$\Delta^-(t-\delta)\geq HD(K(B^T)) = HD(K(B))\geq (1-\eta)\Delta^+(t).$$
Also, a ``symmetric'' estimate holds after exchanging the roles of $\Delta^-$ and $\Delta^+$. Hence, $\Delta^-(t)=\Delta^+(t)$. Moreover, the inequality above says that $\Delta^-(t)=\Delta^+(t)$ is a lower-semicontinuous function of $t$. Since we already know that $\Delta^+(t)$ is an upper-semicontinuous function of $t$ thanks to Proposition \ref{p.upper-sc}, we conclude that $t\mapsto \Delta^-(t)=\Delta^+(t)$ is continuous. Finally, by Proposition \ref{p.Lagrange-Cantor}, from $\Sigma(B)\subset\Sigma_{t-\delta}$, we also have that 
$$d(t-\delta)\geq HD(\ell(\Sigma(B))) = \min\{1,2\cdot HD(K(B))\}\geq (1-\eta)\min\{1,2\Delta^+(t)\}$$
Since $d(t)\leq \min\{1,\Delta^+(t)+\Delta^-(t)\}$ (because $\Sigma_t\subset\pi^-(\Sigma_t)\times\pi^+(\Sigma_t)$), the proof is complete. 
\end{proof}

Let us now sketch the construction of the Gauss-Cantor sets $K(B)$ approaching $\Sigma_t$ from inside. 

\begin{proof}[Sketch of proof of Theorem \ref{t.lower-sc}] Fix $r_0\in\mathbb{N}$ large enough so that 
$$\left|\frac{\log\#C^+(t,r)}{r} - \Delta^+(t)\right|<\frac{\eta}{80}\Delta^+(t)$$ 
for all $r\geq r_0$. 

Set $B_0:=C^+(t,r_0)$, $k=8(\# B_0)^2\lceil 80/\eta\rceil$ and 
$$\widetilde{B}:=\{\beta=(\beta_1,\dots,\beta_k):\beta_j\in B_0 \textrm{ and } I^+(\beta)\cap K_t^+\neq\emptyset\}\subset B_0^k$$ 

It is not hard to show that $\widetilde{B}$ has a significant cardinality in the sense that 
$$\#\widetilde{B} > 2 (\# B_0)^{(1-\tfrac{\eta}{40})k}$$ 
In particular, one can use this information to prove that $HD(K(\widetilde{B}))$ is not far from $\Delta^+(t)$, i.e. 
$$HD(K(\widetilde{B}))\geq (1-\frac{\eta}{20})\Delta^+(t)$$

Unfortunately, since we have no control on the values of $m$ on $\Sigma(\widetilde{B})$, there is no guarantee that $\Sigma(\widetilde{B})\subset\Sigma_{t-\delta}$ for some $\delta>0$. 

We can overcome this issue with the aid of the notion of \emph{left-good} and \emph{right-good} positions. More concretely, we say that $1\leq j\leq k$ is a right-good position of $\beta=(\beta_1,\dots,\beta_k)\in\widetilde{B}$ whenever there are two elements $\beta^{(s)}=\beta_1\dots\beta_j\beta_{j+1}^{(s)}\dots\beta_k^{(s)}\in\widetilde{B}$, $s\in\{1,2\}$ such that 
$$[0;\beta_j^{(1)}]<[0;\beta_j]<[0;\beta_j^{(2)}]$$ 
Similarly, $1\leq j\leq k$ is a left-good position $\beta=(\beta_1,\dots,\beta_k)\in\widetilde{B}$ whenever there are two elements $\beta^{(s)}=\beta_1\dots\beta_j\beta_{j+1}^{(s)}\dots\beta_k^{(s)}\in\widetilde{B}$, $s\in\{3,4\}$ such that 
$$[0;(\beta_j^{(3)})^T]<[0;\beta_j^T]<[0;(\beta_j^{(2)})^T]$$ 
Furthermore, we say that $1\leq j\leq k$ is a \emph{good position}  of $\beta=(\beta_1,\dots,\beta_k)\in\widetilde{B}$ when it is both a left-good and a right-good position. 

Since there are at most two choices of $\beta_j\in B_0$ when $\beta_1,\dots,\beta_{j-1}$ are fixed and $j$ is a right-good position, one has that the subset 
$$\mathcal{E}:=\{\beta\in\widetilde{B}:\beta \textrm{ has } 9k/10 \textrm{ good positions (at least)}\}$$ of \emph{excellent} words in $\widetilde{B}$ has cardinality 
$$\#\mathcal{E} > \frac{1}{2} \#\widetilde{B} > (\# B_0)^{(1-\tfrac{\eta}{40})k}$$ 

We \emph{expect} the values of $m$ on $\Sigma(\mathcal{E})$ to \emph{decrease} because excellent words have many good positions. Also, the Hausdorff dimension of $K(\mathcal{E})$ is not far from $\Delta^+(t)$ thanks to the estimate above on the cardinality of $\mathcal{E}$. However, there is no reason for $\Sigma(\mathcal{E})\subset\Sigma_{t-\delta}$ for some $\delta>0$ because an \emph{arbitrary} concatenation of words in $\mathcal{E}$ might not belong to $\Sigma_t$. 

At this point, the idea is to build a complete shift $\Sigma(B)\subset\Sigma_{t-\delta}$ from $\mathcal{E}$ with the following combinatorial argument. Since $\beta=(\beta_1,\dots,\beta_k)\in\mathcal{E}$ has $9k/10$ good positions, we can find good positions $1\leq i_1\leq i_2\leq\dots\leq i_{\lceil 2k/5\rceil}\leq k-1$ such that $i_s+2\leq i_{s+1}$ for all $1\leq s\leq\lceil 2k/5\rceil-1$ and $i_s+1$ are also good positions for all $1\leq s\leq \lceil 2k/5\rceil$. Because $k:=8(\# B_0)^2\lceil 80/\eta\rceil$, the pigeonhole principle reveals that we can choose positions $j_1\leq \dots\leq j_{3(\# B_0)^2}$ and words $\widehat{\beta}_{j_1}, \widehat{\beta}_{j_1+1},\dots, \widehat{\beta}_{j_{3(\# B_0)^2}}, \widehat{\beta}_{j_{3(\# B_0)^2}+1}\in B_0$ such that $j_s+2\lceil 80/\eta\rceil\leq j_{s+1}$ for all $s<3(\# B_0)^2$ and the subset 
$$X=\{(\beta_1,\dots,\beta_k)\in\mathcal{E}: j_s, j_s+1 \textrm{ are good positions and } \beta_{j_s}=\widehat{\beta}_{j_s}, \beta_{j_s+1}=\widehat{\beta}_{j_s+1} \,\forall\,\,s\leq 3(\# B_0)^2 \}$$ 
of excellent words with prescribed subwords $\widehat{\beta}_{j_s}$, $\widehat{\beta}_{j_s+1}$ at the good positions $j_s$, $j_s+1$ has cardinality 
$$\#X > (\# B_0)^{(1-\tfrac{\eta}{20})k}$$ 
Next, we convert $X$ into the alphabet $B$ of an appropriate complete shift with the help of the projections $\pi_{a,b}:X\to B_0^{j_b-j_a}$, $\pi_{a,b}(\beta_1,\dots,\beta_k) = (\beta_{j_a+1},\beta_{j_a+2},\dots,\beta_{j_b})$. More precisely, an elementary counting argument shows that we can take $1\leq a<b\leq 3(\# B_0)^2$ such that $\widehat{\beta}_{j_a}=\widehat{\beta}_{j_b}$, $\widehat{\beta}_{j_a+1} = \widehat{\beta}_{j_b+1}$, and the image $\pi_{a,b}(X)$ of some projection $\pi_{a,b}$ has a significant cardinality  
$$\#\pi_{a,b}(X) > (\# B_0)^{(1-\tfrac{\eta}{4})(j_b-j_a)}$$ 
From these properties, we get an alphabet $B=\pi_{a,b}(X)$ whose words concatenate in an appropriate way (because $\widehat{\beta}_{j_a}=\widehat{\beta}_{j_b}$, $\widehat{\beta}_{j_a+1} = \widehat{\beta}_{j_b+1}$), the Hausdorff dimension of $K(B)$ is $HD(K(B))>(1-\eta)\Delta^+(t)$ (because $\# B >(\# B_0)^{(1-\tfrac{\eta}{4})(j_b-j_a)}$ and $j_b-j_a>2\lceil\tfrac{80}{\eta}\rceil$), and $\Sigma(B)\subset\Sigma_{t-\delta}$ for some $\delta>0$ (because the features of good positions forces the values of $m$ on $\Sigma(B)$ to decrease). This completes our sketch of proof. 
\end{proof}

\subsection{End of proof of Moreira's theorem \ref{t.Gugu}}

By Corollary \ref{c.Gugu-continuity}, the function 
$$t\mapsto d(t)=HD(L\cap (-\infty,t))$$ 
is continuous. Moreover, an inspection of the proof of Corollary \ref{c.Gugu-continuity} shows that we have also proved the equality $HD(M\cap(-\infty,t)) = HD(L\cap(-\infty,t))$. 

Therefore, our task is reduced to prove that $d(3+\varepsilon)>0$ for all $\varepsilon>0$ and $d(\sqrt{12})=1$. 

The fact that $d(3+\varepsilon)>0$ for any $\varepsilon$ uses explicit sequences $\theta_m\in\{1,2\}^{\mathbb{Z}}$ such that $\lim\limits_{m\to\infty} m(\theta_m) = 3$ in order to exhibit non-trivial Cantor sets in $M\cap (-\infty,3+\varepsilon)$. More precisely, consider\footnote{This choice of $\theta_m$ is motivated by the discussion in Chapter 1 of Cusick-Flahive book \cite{CF}.} the periodic sequences $$\theta_m:=\overline{2\underbrace{1\dots 1}_{2m \textrm{ times}} 2}$$ 
where $\overline{a_1\dots a_k}:=\dots a_1\dots a_k \,\, a_1\dots a_k\dots$. Since the sequence $\theta_{\infty} = \overline{1}, 2, 2, \overline{1}$ has the property that $m(\theta_{\infty}) = [2; \overline{1}]+[0;2,\overline{1}] =3$, and $|[a_0;a_1,\dots, a_n, b_1,\dots]-[a_0;a_1,\dots,a_n,c_1,\dots]|<\frac{1}{2^{n-1}}$ in general\footnote{See Lemma 2 in Chapter 1 of \cite{CF}.}, we have that the alphabet $B_m$ consisting of the two words $2\underbrace{1\dots 1}_{2m \textrm{ times}} 2$ and $2\underbrace{1\dots 1}_{2m+2 \textrm{ times}} 2$ satisfies 
$$\Sigma(B_m)\subset \Sigma_{3+\frac{1}{2^m}}$$ 
Thus, $d(3+\tfrac{1}{2^m})=HD(M\cap(-\infty, 3+\frac{1}{2^m}))\geq HD(\Sigma(B_m)) = 2\cdot HD(K(B_m))>0$ for all $m\in\mathbb{N}$. 

Finally, the fact that $d(\sqrt{12})=1$ follows from Corollary \ref{c.Gugu-continuity} and Remark \ref{r.JP}. Indeed, Perron showed that $m(\theta)\leq\sqrt{12}$ if and only if $\theta\in\{1,2\}^{\mathbb{Z}}$ (see the proof of Lemma 7 in Chapter 1 of Cusick-Flahive book \cite{CF}). Thus, $K_{\sqrt{12}}^+ = C(2)$. By Corollary \ref{c.Gugu-continuity}, it follows that 
$$d(\sqrt{12})=\min\{1,2\cdot \Delta^+(\sqrt{12})\} = \min\{1,2\cdot HD(C(2))\}$$ 
Since Remark \ref{r.JP} tells us that $HD(C(2))>1/2$, we conclude that $d(\sqrt{12})=1$. 

%%%%%%%%%%%%%%%%%%%%%%%%%%%%%%%%%%%
%%%%%%%%%%%%%%%%%%%%%%%%%%%%%%%%%%%
%%%%%%%%%%%%%%%%%%%%%%%%%%%%%%%%%%%
%%%%%%%%%%%%%%%%%%%%%%%%%%%%%%%%%%%
%%%%%%%%%%%% Appendix %%%%%%%%%%%%%%%%
%%%%% Hurwitz theorem, Euler remark  %%%%%%
%%%%%%%%%%%%%%%%%%%%%%%%%%%%%%%%%%%
%%%%%%%%%%%%%%%%%%%%%%%%%%%%%%%%%%%
%%%%%%%%%%%%%%%%%%%%%%%%%%%%%%%%%%%
%%%%%%%%%%%%%%%%%%%%%%%%%%%%%%%%%%%

\appendix

\section{Proof of Hurwitz theorem}\label{a.Hurwitz}

Given $\alpha\notin\mathbb{Q}$, we want to show that the inequality 
$$\left|\alpha-\frac{p}{q}\right|\leq\frac{1}{\sqrt{5}q^2}$$ 
has infinitely many rational solutions. 

In this direction, let $\alpha=[a_0;a_1,\dots]$ be the continued fraction expansion of $\alpha$ and denote by $[a_0;a_1,\dots,a_n] = p_n/q_n$. We affirm that, for every $\alpha\notin\mathbb{Q}$ and every $n\geq 1$, we have 
$$\left|\alpha-\frac{p}{q}\right|<\frac{1}{\sqrt{5}q^2}$$ 
for some $\frac{p}{q}\in\{\frac{p_{n-1}}{q_{n-1}}, \frac{p_n}{q_n}, \frac{p_{n+1}}{q_{n+1}}\}$. 

\begin{remark} Of course, this last statement provides infinitely many solutions to the inequality $\left|\alpha-\frac{p}{q}\right|\leq\frac{1}{\sqrt{5}q^2}$. So, our task is reduced to prove the affirmation above. 
\end{remark} 

The proof of the claim starts by recalling Perron's Proposition \ref{p.Perron}:
$$\alpha-\frac{p_n}{q_n} = \frac{(-1)^n}{(\alpha_{n+1}+\beta_{n+1})q_n^2}$$ 
where $\alpha_{n+1}:=[a_{n+1};a_{n+2},\dots]$ and $\beta_{n+1} = 
\frac{q_{n-1}}{q_n} = [0;a_n,\dots,a_1]$. 

For the sake of contradiction, suppose that the claim is false, i.e., there exists $k\geq 1$ such that \begin{equation}\label{e.A1}
\max\{(\alpha_k+\beta_k), (\alpha_{k+1}+\beta_{k+1}), (\alpha_{k+2}+\beta_{k+2})\}\leq \sqrt{5}
\end{equation}

Since $\sqrt{5}<3$ and $a_m\leq\alpha_m+\beta_m$ for all $m\geq 1$, it follows from \eqref{e.A1} that 
\begin{equation}\label{e.A2}
\max\{a_k,a_{k+1},a_{k+2}\}\leq 2
\end{equation} 

If $a_m=2$ for some $k\leq m\leq k+2$, then \eqref{e.A2} would imply that $\alpha_m+\beta_m\geq 2+[0;2,1] = 2+\frac{1}{3}>\sqrt{5}$, a contradiction with our assumption \eqref{e.A1}. 

So, our hypothesis \eqref{e.A1} forces 
\begin{equation}\label{e.A3}
a_k=a_{k+1}=a_{k+2}=1
\end{equation} 

Denoting by $x=\frac{1}{\alpha_{k+2}}$ and $y=\beta_{k+1} = q_{k-1}/q_k\in\mathbb{Q}$, we have from \eqref{e.A3} that 
$$\alpha_{k+1}=1+x, \quad \alpha_k = 1+\frac{1}{1+x}, \quad \beta_{k} = \frac{1}{y}-1, \quad \beta_{k+2} = \frac{1}{1+y}$$

By plugging this into \eqref{e.A1}, we obtain  
\begin{equation}\label{e.A4}
\max\left\{\frac{1}{1+x}+\frac{1}{y}, 1+x+y, \frac{1}{x}+\frac{1}{1+y}\right\}\leq \sqrt{5}
\end{equation}

On one hand, \eqref{e.A4} implies that 
$$\frac{1}{1+x}+\frac{1}{y}\leq \sqrt{5} \quad \textrm{and} \quad 1+x\leq \sqrt{5}-y.$$
Thus, 
$$\frac{\sqrt{5}}{y(\sqrt{5}-y)} = \frac{1}{\sqrt{5}-y}+\frac{1}{y}\leq \frac{1}{1+x}+\frac{1}{y}\leq\sqrt{5},$$
and, \emph{a fortiori}, $y(\sqrt{5}-y)\geq 1$, i.e., 
\begin{equation}\label{e.A5}
\frac{\sqrt{5}-1}{2}\leq y\leq \frac{\sqrt{5}+1}{2}
\end{equation}

On the other hand, \eqref{e.A4} implies that 
$$x\leq \sqrt{5}-1-y \quad \textrm{and} \quad \frac{1}{x}+\frac{1}{1+y}\leq \sqrt{5}.$$ 
Hence, 
$$\frac{\sqrt{5}}{(1+y)(\sqrt{5}-1-y)} = \frac{1}{\sqrt{5}-1-y}+\frac{1}{1+y}\leq \frac{1}{x}+\frac{1}{1+y}\leq\sqrt{5},$$
and, \emph{a fortiori}, $(1+y)(\sqrt{5}-1-y)\geq 1$, i.e., 
\begin{equation}\label{e.A6}
\frac{\sqrt{5}-1}{2}\leq y\leq \frac{\sqrt{5}+1}{2}
\end{equation}

It follows from \eqref{e.A5} and \eqref{e.A6} that $y=(\sqrt{5}-1)/2$, a contradiction because $y=\beta_{k+1}= q_{k-1}/q_k\in\mathbb{Q}$. This completes the argument. 

\section{Proof of Euler's remark}\label{a.Euler}

Denote by $[0; a_1, a_2,\dots, a_n] = \frac{p(a_1,\dots,a_n)}{q(a_1,\dots,a_n)} = \frac{p_n}{q_n}$. It is not hard to see that  
$$q(a_1)=a_1, \quad q(a_1,a_2) = a_1a_2+1, \quad q(a_1,\dots,a_n) = a_n q(a_1,\dots,a_{n-1}) + q(a_1,\dots,a_{n-2}) \,\,\,\,\forall\,\,n\geq 3.$$

From this formula, we see that $q(a_1,\dots,a_n)$ is a sum of the following products of elements of $\{a_1,\dots,a_n\}$. First, we take the product $a_1\dots a_n$ of all $a_i$'s. Secondly, we take all products obtained by removing any pair $a_i a_{i+1}$ of adjacent elements. Then, we iterate this procedure until no pairs can be omitted (with the convention that if $n$ is even, then the empty product gives $1$). This rule to describe $q(a_1,\dots,a_n)$ was discovered by Euler. 

It follows immediately from Euler's rule that $q(a_1,\dots,a_n) = q(a_n,\dots,a_1)$. This proves Proposition \ref{p.Euler}. 

%\subsection{Representation theory of $SL(2,\mathbb{R})$ and Ratner's work}

\bibliographystyle{amsplain}

\end{document}